\documentclass[11pt]{article}
\usepackage[utf8]{inputenc}  
\usepackage[T1]{fontenc}
\usepackage{amsmath}
\usepackage{amsfonts,amssymb,verbatim}
\usepackage{amsthm}
\usepackage{mathrsfs}
\usepackage{array}
\usepackage{multicol}
\usepackage{moreverb}
\usepackage{pdfpages}
\usepackage{enumerate}
\usepackage{xcolor}
\usepackage{indentfirst}
\usepackage{caption}
\usepackage{hyperref}  
\usepackage{pgf,tikz}
\usepackage{mathrsfs}

\usetikzlibrary{arrows}               
 \hypersetup{
    hyperfigures = true,
    colorlinks = true,
    linkcolor=blue
    }
 
 \usepackage[hmargin={3cm,3cm},top=0.8cm,bottom=3.8cm]{geometry}
\usepackage{charter}
\usepackage{graphicx}
\usepackage{aeguill,lmodern}\usepackage{bclogo}

\usepackage{fancyhdr}

\theoremstyle{theorem}
\newtheorem{thmt}{Theorem}[section]
\newtheorem{lem}[thmt]{Lemma}
\newtheorem{prop}[thmt]{Proposition}
\newtheorem{defi}[thmt]{Definition}

\newtheorem{rem}{Remark}[section]
\newtheorem*{lem*}{Lemme}
\newtheorem{cor}[thmt]{Corollary}
\newtheorem{thm}{Theorem} 

\title{}
\date{}

\makeatletter
\DeclareRobustCommand{\cev}[1]{%
  \mathpalette\do@cev{#1}%
}
\newcommand{\do@cev}[2]{%
  \fix@cev{#1}{+}%
  \reflectbox{$\m@th#1\vec{\reflectbox{$\fix@cev{#1}{-}\m@th#1#2\fix@cev{#1}{+}$}}$}%
  \fix@cev{#1}{-}%
}
\newcommand{\fix@cev}[2]{%
  \ifx#1\displaystyle
    \mkern#23mu
  \else
    \ifx#1\textstyle
      \mkern#23mu
    \else
      \ifx#1\scriptstyle
        \mkern#22mu
      \else
        \mkern#22mu
      \fi
    \fi
  \fi
}

\makeatother

\def\N{{\mathbb N}}
\def\R{{\mathbb R}}
\def\C{{\mathbb C}}

\def\P{{\mathbb P}}
\def\E{{\mathbb E}}

\newtheoremstyle{ouech}
{\topsep}
{\topsep}
{\upshape} 
{} 
{\bfseries} 
{} 
{\newline}
{\thmname{#1}\thmnumber{ #2}\thmnote{#3}} 
\theoremstyle{ouech}

\theoremstyle{exemple}

\theoremstyle{ouech}

\theoremstyle{ouech}

\numberwithin{equation}{section}

\begin{document}

\title{Invariant measures of critical branching random walks in high dimension}
\author{V.~Rapenne}
\maketitle
\begin{center}
Institut Camille Jordan\\
Université Claude Bernard Lyon 1
\end{center}
\begin{abstract}
In this work, we characterize cluster-invariant point processes for critical branching spatial processes on $\R^d$ for all large enough $d$ when the motion law is $\alpha$-stable or has a finite discrete range. More precisely, when the motion is $\alpha$-stable with $\alpha\leq 2$ and the offspring law $\mu$ of the branching process has an heavy tail such that $\mu(k)\sim k^{-2-\beta}$, then we need the dimension $d$ to be strictly larger than the critical dimension $\alpha/\beta$. In particular, when the motion is Brownian and the offspring law $\mu$ has a second moment, this critical dimension is 2. Contrary to the previous work of Bramson, Cox and Greven in \cite{BCG} whose proof used PDE techniques, our proof uses probabilistic tools only.
\end{abstract}
\vspace{0.5 cm}
\vspace{1.5 cm}

\section{Introduction}
\subsection{Definition of the model and first notation} \label{debut}
Let $\mu$ be a probability distribution on $\N$ called the "offspring law". We assume that $\mu$ is critical, that is, $\sum_{k=0}^{+\infty} k\mu(k)=1$. Let $\mathcal{P}$ be a probability distribution on $\R^d$ with $d\in\N^*$. $\mathcal{P}$ is called the "motion law". We define a discrete-time critical branching process on $\R^d$ in the following way:

We start with one particle at an initial position $x\in\R^d$. It is generation 0 of the process. Let $\mathcal{Z}_n$ be the set of particles at generation $n$. Every particle $u\in\mathcal{Z}_n$ gives birth independently to $N_u$ offsprings where $N_u$ has law $\mu$. The offsprings of $u$ jump independently of each others from the position of $u$ according to the motion law $\mathcal{P}$. All the offsprings of $\mathcal{Z}_n$ with their new positions form the $(n+1)$-th generation of the branching process. We remark that the underlying  genealogical tree $\mathbf{T}_{gen}$ of the model is a Galton-Watson tree with offspring law $\mu$. As $\mu$ is critical, $\mathbf{T}_{gen}$ will dies out a.s. For every branching process defined in this article, we use the following notation:
$$\begin{array}{ll}
|u|&:= \text{ the generation of some particle }u\\
Z_n&:=|\mathcal{Z}_n|=\text{ the total number of particles in generation }n\\
\{S_u:|u|=n\}&:=\text{ the set of particle positions in generation }n\\
Z_n(A)&:=\text{ the number of particles in the set A in generation }n
\end{array}$$

A special case of this model is when $\mathcal{P}$ is a standard Gaussian distribution. In this situation the model is called the critical Wiener branching process. One can refer to \cite{Revesz96}, \cite{Revesz98} and \cite{Revesz02} for more properties of the critical Wiener branching process. Another particular case is when $\mathcal{P}$ corresponds to the simple random walk on $\mathbb{Z}^d$. In \cite{LZ}, Lalley and Zheng studied this critical branching simple random walk conditioned on long survival when $d\geq 2$. In particular, they prove a spatial generalization of the famous Yaglom's theorem when $d\geq 3$.

In the previous paragraph, we defined a critical branching process starting from a single particle. However we can also start from a point process in $\R^d$.
Let us consider a locally finite point process $$\Theta=\sum\limits_{i\in I} \delta_{x_i}$$ where $I$ is finite or countable. For every point $x_i$, it is possible to define a critical branching process with critical offspring law $\mu$ and motion law $\mathcal{P}$ starting from an ancestor located at $x_i$. For every $i\in I$ and every $n\in\N$, we denote by $\{S_u^{(i)},|u|=n\}$ the set of positions of the $n$-th generation of the branching process starting from an ancestor located at $x_i$. We assume that these branching processes are independent of each other. This collection of critical branching processes gives us a sequence $(\Theta_n)_{n\in\N}$ of point processes on $\R^d$ . More precisely $\Theta_0=\Theta$ and at each time $n\geq 1$,
$$\Theta_n=\sum\limits_{i\in I} \sum\limits_{\substack{|u|=n}}\delta_{S^{(i)}_u}.$$
We have to be careful here because it is possible that for some $n$, $\Theta_n$ is not locally finite anymore. For example, let us consider a critical Wiener branching process. If we start with a Poisson point process $\Theta_0$ of intensity $e^{||x||^3}dx$, even $\Theta_1$ is not a point process anymore. However $(\Theta_n)_{n\in\N}$ is a sequence of locally finite point processes when $\Theta$ is a Poisson point process of constant intensity. Indeed in this situation, if $A$ is a closed ball of $\R^d$, $\E\left[\Theta_n(A)\right]$ remains constant and proportional to the volume of $A$ because $\mu$ is critical.
One can refer to \cite{Kallenberg_modern} or \cite{MKM} for more information on point processes. Moreover, we offer a short warm-up about point processes in subsection \ref{warmup} for this article to be self-contained.

A critical branching process starting from a point process is a special case of a more general theory known as the theory of cluster fields. This theory was originally developped by Liemant, Kerstan, Matthes and Prehn in several papers (in German), e.g \cite{Liemant} and \cite{KMP}. The book \cite{MKM} summarizes most of important facts about this theory. (Especially chapters 11 and 12.) A very important issue about point processes is to know whether $(\Theta_n)_{n\in\N}$ converges toward a non-trivial point process or not. This is strongly linked with the concept of stability which is studied in \cite{kallenberg_article}. It is well-known that in dimension 1 or 2, if the motion $\mathcal{P}$ is Gaussian, $(\Theta_n)_{n\in\N}$ converges in law in the vague sense toward the trivial null measure.  (See for example sections 7 and 8 of \cite{Dur} in a continuous time setting.) On the contrary, in dimension $d\geq 3$, $$\Theta_n\xrightarrow[n\rightarrow+\infty]{law}\Lambda$$ where $\Lambda$ is non-trivial. (See \cite{DI} in a continuous-time setting.) This limiting point process $\Lambda$ is cluster-invariant, that is, if $\Theta_0=\Lambda$, then $\Theta_n$ is distributed as $\Lambda$ for every $n\in\N^*$. This raises the natural question of classification of all cluster-invariant point processes. If we assume spatial stationarity  or boundedness of the intensity of the initial process $\Theta_0$, the classification of cluster-invariant point processes is a well-known fact in a very general setting. For example, on can refer to Theorem 12.4.6 in \cite{MKM}. However, spatial stationarity  is a very strong hypothesis. The classification of cluster-invariant point processes for critical branching Brownian motion is treated without any extra assumption like spatial stationarity in \cite{BCG}. Their proof is given in a continuous time setting and is based on PDE's techniques. Our article aims to characterize all cluster-invariant point processes in a more general setting by using only probabilistic tools. We also mention that the method which is used in this article is inspired by \cite{chen} which treats 1-dimensional binary branching Brownian motion with critical drift.

\subsection{Previous results}
\label{previous}
Let us recall the known result on critical branching Brownian motion in continuous time which is defined as follows:
\begin{itemize}
\item It starts with one particle.
\item Each particle dies at rate 1.
\item At the end of its life, a particle is replaced by 0 particle with probability $1/2$ and by 2 particles with probability $1/2$.
\item During its lifetime, a particle moves like a $d$-dimensional standard Brownian motion.
\end{itemize}
Let us summarize the following results of \cite{Dur}, \cite{DI} and \cite{BCG}. 
\begin{thm}[\cite{Dur}, \cite{DI} and \cite{BCG}] \label{originalth}
If we consider a critical branching Brownian motion starting from a Poisson point process $\Theta_0$ of constant intensity $\theta$, it induces a continuous time family of point processes $(\Theta_t)_{t\geq 0}$ by collecting the positions of the alive particles at time $t$.
Then we have the following results:
\begin{enumerate}[(i)]
\item If $d\in\{1,2\}$, then $(\Theta_t)_{t\geq 0}$ converges vaguely toward the null point process.
\item If $d\geq 3$, then $(\Theta_t)_{t\geq 0}$ converges vaguely toward a non-trivial point process $\tilde{\Lambda}^{d,\theta}_{\infty}$.
\end{enumerate}
Furthermore, it characterizes the set of cluster-invariant point processes, that is, point processes $\Theta$ such that $\Theta_t\overset{law}=\Theta$ for every $t\in\R_+$ if $\Theta_0=\Theta$. Indeed, $\{\mathcal{L}\left(\Theta \right), \Theta\text{ is cluster-invariant}\}$ is a convex set whose extremal points are $$\{\mathcal{L}\left(\tilde{\Lambda}^{d,\theta}_{\infty} \right), \theta\in\R_+\}.$$
\end{thm}
\begin{rem}
The first point (i) stems from sections 7 and 8 of \cite{Dur}. The second point (ii) is Theorem 7.1 in \cite{DI}. Moreover, characterization of cluster-invariant point processes is Theorem 1 in \cite{BCG}.
\end{rem}
\subsection{Main results}
Now let us state our results which generalize Theorem \ref{originalth} in the discrete time setting. Indeed we work with a general critical offspring law $\mu$ and a quite general motion law $\mathcal{P}$. We consider the following three cases for $\mathcal{P}$  and $\mu$ in this paper.

\begin{enumerate}
\item Hypothesis $\mathcal{H}_1$: The distribution $\mathcal{P}$ is a standard $d$-dimensional Gaussian random variable with $d\geq 3$. Moreover $\sigma^2:=\sum_{k=0}^{+\infty}k^2\mu(k)<+\infty$. Further, we write $\Sigma=I_d$, where $I_d$ is the identity matrix.
\item Hypothesis $\mathcal{H}_2$: We assume that $\mathcal{P}$ can be written as $$\mathcal{P}=\sum_{x\in \mathcal{R}}p(x)\delta_x$$ where $\mathcal{R}$ is a finite subset of $\mathbb{Z}^d$ with $d\geq 3$. Moreover $p$ is symmetrical in the sense that for every $x\in\mathcal{R}$, $-x\in\mathcal{R}$ and $p(-x)=p(x)$. $\mathcal{P}$ has a positive definite covariance matrix $\Sigma$. In addition, we assume that the random walk generated by the motion $\mathcal{P}$ is aperiodic. Furthermore $\sigma^2:=\sum_{k=0}^{+\infty}k^2\mu(k)<+\infty$.
\item Hypothesis $\mathcal{H}_3$:  The motion law $\mathcal{P}$ is given by a spherically symmetric $\alpha$-stable law with $\alpha\in]0,2[$. More precisely for every $y\in\R^d$, $$\int\exp\left(i\langle y,x \rangle \right)\mathcal{P}(dx)=\exp\left(-\left(\sum\limits_{k=1}^d|y_k|^2\right)^{\alpha/2}\right).$$ The critical offspring law $\mu$ has no second moment anymore. However we assume that there exists $\beta\in]0,1]$ such that for every $\gamma<\beta$, $\sum_{k=0}^{+\infty} k^{1+\gamma}\mu(k)<+\infty$. Moreover, we assume that $d> \alpha/\beta$.
\end{enumerate}

Here, hypothesis $\mathcal{H}_2$ is a generalization of hypothesis $\mathcal{H}_1$. Indeed, in both cases, $\mathcal{P}$ has a second moment. Consequently it is in the domain of attraction of a Gaussian distribution. However, having a second moment is not enough for us. Indeed, in our arguments, we need the random walk associated to $\mathcal{P}$ to meet the heat-kernel estimate stated in Proposition \ref{heat}. Our results hold for any distribution $\mathcal{P}$ which has a finite second moment and satisfies this heat-kernel estimate. For now, we do not know how to remove this technical assumption. Under hypotheses $\mathcal{H}_1$ and $\mathcal{H}_2$, as $\mathcal{P}$ is in the domain of attraction of a Gaussian distribution, the critical branching process with motion law $\mathcal{P}$ is analoguous to a binary branching Brownian motion. In particular, assuming $\mathcal{H}_1$ or $\mathcal{H}_2$, the critical dimension is always 2 as in Theorem \ref{originalth}. On the contrary, under hypothesis $\mathcal{H}_3$, $\mathcal{P}$ is no more in the Brownian domain. This leads to a change of the critical dimension, among other consequences. Indeed, assuming hypothesis $\mathcal{H}_3$, Gorostiza and Wakolbinger showed in \cite{GW} that the critical dimension is $\alpha/\beta$.

Now, let $X$ be a non-negative random variable. If we assume hypotheses $\mathcal{H}_1$ or $\mathcal{H}_3$, let $\Lambda_0^{d,X}$ be a Poisson point process with distribution $$PPP(X\lambda(dx))$$ where $\lambda$ is the Lebesgue measure. If we assume hypothesis $\mathcal{H}_2$, let $\Lambda_0^{d,X}$ be the discrete Poisson point process $$PPP\left(X\sum\limits_{x\in\mathbb{Z}^d}\delta_x\right).$$
As in subsection \ref{debut}, we can define the sequence $(\Lambda_n^{d,X})_{n\in\N}$ by considering the critical branching process starting from $\Lambda_0^{d,X}$. In this paper, a "closed ball" designates a euclidean closed ball of $\R^d$. Under hypothesis $\mathcal{H}_2$, we always assume that a closed ball contains at least one point of $\mathbb{Z}^d$.
\begin{thm}[Convergence Theorem]\label{cv}
We assume hypotheses $\mathcal{H}_1$ or $\mathcal{H}_2$. Then, there exists a non-trivial point process $\Lambda^{d,X}_{\infty}$ such that\\
$$ \Lambda_n^{d,X}\xrightarrow[n\rightarrow+\infty]{law}\Lambda^{d,X}_{\infty}$$
in the vague topology.
In addition, the point process $\Lambda^{d,X}_{\infty}$ can be described in the following way on every closed ball $A$:\\
$$\Lambda_{\infty}^{d,X}(\cdot\cap A)=\sum\limits_{i=1}^{P_A}\mathcal{N}_A^{(i)} $$
where $(\mathcal{N}_A^{(i)})_{i\geq 1}$ are i.i.d copies of a point process $\mathcal{N}_A$ defined in Proposition \ref{limitpp} and $P_A$ is a Poisson random variable of parameter $ X I_A$ where $I_A$ is a constant defined in definition \ref{defiprat}. $P_A$ and $(\mathcal{N}_A^{(i)})_{i\geq 1}$ are assumed to be independent.

\end{thm}
\begin{rem}
Actually, the heat kernel estimate in $\mathcal{H}_2$ is not really necessary for Theorem \ref{cv}. We see in the proof of Lemma \ref{intensity} that we only need the existence of a positive integrable function $g$ such that for every $x\in\mathbb{Z}^d$ and for every $n\in\N^*$, $\P(\hat{S}_n=x)\leq n^{-d/2}g(n^{-1/2}x)$ where $\hat{S}$ is the random walk associated to $\mathcal{P}$. This is possible as soon as $\mathcal{P}$ has a moment of order $d+1$ thanks to the asymptotic expansion of the local-limit theorem. (See Theorem 22.1 in \cite{Rao}.) However, assumption $\mathcal{H}_2$ is crucial in the proof of Theorem \ref{charac} for technical reasons.
\end{rem}
A point process $\Theta$ is said to be cluster-invariant if the sequence $(\Theta_n)_{n\in\N}$ obtained by the critical branching process starting from $\Theta_0=\Theta$ as in subsection \ref{debut} satisfies $$\Theta_n\overset{law}=\Theta$$ for every $n\in\N$. One sees easily that the limiting point processes obtained in Theorems \ref{originalth} and \ref{cv} are cluster-invariant. In fact, all cluster-invariant point processes are given by the following theorem:
\begin{thm}[Characterization of cluster-invariant measures] \label{charac}
Let us assume hypotheses $\mathcal{H}_1$ or $\mathcal{H}_2$.
Let $\Theta$ be a cluster-invariant point process. Then there exists a non-negative random variable $X$ such that:
$$\Theta\overset{law}=\Lambda^{d,X}_{\infty}. $$
Moreover $\Lambda^{d,X}_{\infty}$ is cluster-invariant for every non-negative random variable $X$.
\end{thm}

Theorems \ref{cv} and \ref{charac} assumed hypothesis $\mathcal{H}_1$ or hypothesis $\mathcal{H}_2$. Under these hypotheses, the proofs are quite similar. Under the hypothesis $\mathcal{H}_3$, the results of Theorems \ref{cv} and \ref{charac} remain true. However proofs require some slight modifications. We focus on these modifications in section \ref{gene}.
\begin{thm}\label{general}
Let us assume hypothesis $\mathcal{H}_3$. Then the conclusions of Theorems \ref{cv} and \ref{charac} remain true.
\end{thm}
\subsection{Organisation of the paper}
\begin{itemize}
\item In section \ref{sec_une_particule}, we begin by recalling some results about local limit theorems, heat kernel estimates and point processes. In subsection \ref{measure}, we quickly explain the classical "spine method" which is a fundamental tool in this article. Finally, in subsection \ref{secfunlemma}, we study the branching process starting from a single particle conditionally on survival in a given set.
\item In section \ref{seccv}, we prove Theorem \ref{cv}.
\item Theorem \ref{cv} gives us the point process $\Lambda^{d,X}_{\infty}$ as a limiting object. We give an independent proof of the compatibility  $\Lambda^{d,X}_{\infty}$ in section \ref{sec4}.
\item In section \ref{sec5}, we prove Theorem \ref{charac}.
\item Then, we briefly explain how to prove analoguous versions of Theorems \ref{cv} and \ref{charac} under the hypothesis $\mathcal{H}_3$ in section \ref{gene}.
\item We finish in section \ref{sec7} by a short discussion about our results.
\end{itemize}
\section{Preliminaries} \label{sec_une_particule}
\subsection{Local-limit Theorem}
The motion law $\mathcal{P}$ is not always a Gaussian random variable. However all our computations are easier in this case. The local-limit Theorem is a fundamental tool in order to make a link between general random walks and Gaussian random walks. 
\begin{prop}[Local-limit Theorem, Theorem 2.1.1 of \cite{lawler} ]\label{loclimitlat}
Let us assume that $\mathcal{P}$ satisfies hypothesis $\mathcal{H}_2$. Let $(\hat{S}_n)_{n\in\N}$ be a random walk with motion law $\mathcal{P}$ starting from $0$. Then we have the following uniform convergence:
$$\underset{n\rightarrow +\infty}\lim \hspace{0.1 cm}\underset{x\in\mathbb{Z}^d }\sup\hspace{0.1 cm} n^{d/2}\Bigg|\P(\hat{S}_n=x)-\frac{1}{(2\pi)^{d/2}\det(\Sigma)^{1/2}}\times \exp\left( -\frac{1}{2n}\langle x,\Sigma^{-1}x\rangle\right)\Bigg|=0.$$
\end{prop}
Remark that hypothesis $\mathcal{H}_2$ is really strong. Actually we only need the finiteness of the second moment of $\mathcal{P}$ for the local-limit Theorem to be true. (See Theorem 2.3.9 in \cite{lawler}.) Moreover, this theorem implies the following useful corollary.
\begin{cor}\label{inegastand}
Let $(\hat{S}_n)_{n\in\N}$ be a random  walk whose motion law satisfies hypotheses $\mathcal{H}_1$ or $\mathcal{H}_2$.
Let $A$ be a closed ball. Then there exists a constant $c_d$, depending only on the dimension $d$ and the motion law $\mathcal{P}$, such that for every $n\in\N^*$,
$$\P(\hat{S}_n\in A)\leq c_d \frac{|A|}{n^{d/2} }$$
where $|A|$ is the  Lebesgue measure of $A$.
\end{cor}
\subsection{Heat kernel estimate}
In addition to the local-limit Theorem, we also need the following so-called heat kernel estimate:
\begin{prop}[Heat kernel estimate] \label{heat}
Let us assume that $\mathcal{P}$ satisfies hypothesis $\mathcal{H}_2$. Let $(\hat{S}_n)_{n\in\N}$ be a random walk with motion law $\mathcal{P}$ starting from $0$. Then there exists a positive constant $C_1$ such that for every $x\in\mathbb{Z}^d$ and for every $n\in\N^*$:
$$\P(\hat{S}_n=x)\leq C_1n^{-d/2}\exp\left(- \frac{||x||^2}{C_1 n} \right). $$
Moreover there exist positive constants $\tau$ and $C_2$ such that for every $n\in\N^*$ and for every $x\in\mathbb{Z}^d$ satisfying $||x||\leq \tau n$:
$$\P(\hat{S}_n=x)\geq C_2n^{-d/2}\exp\left(- \frac{||x||^2}{C_2 n} \right). $$
\end{prop}
\begin{proof}
A random walk $(\hat{S}_n)_{n\in\N}$ with motion law $\mathcal{P}$ under the hypothesis $\mathcal{H}_2$ can be interpreted as a random walk with conductances $(c(x,y))_{x,y\in G}$ where $G$ is a graph whose vertices are given by $\mathbb{Z}^d$ and $(x,y)$ is an edge of $G$ iff $x-y\in \mathcal{R}$. We recall that $\mathcal{R}$ is the finite support of $\mathcal{P}$. For every edge $(x,y)$ in the graph $G$, $c(x,y)=p(y-x)$. For every $R\geq 1$, we define $\mathcal{B}(R)=\{x\in G, d(0,x)\leq R\}$ and $V(R)=|\mathcal{B}(R)|$ where $d$ is the graph distance in $G$.
According to Theorem 3.3.5 and Proposition 3.3.2 of \cite{Kumagai}, two conditions have to be checked in order to satisfy heat kernel estimate: 
\begin{itemize}
\item Condition $(VD)$: There exists a positive constant $C$ such that for every $R\geq 1$,
$$V(R)\leq C V(2R). $$
\item Condition $(WPI(2))$: There exists a positive constant $C'$ such that for every $R\geq 1$ and for every $f:\mathcal{B}(R)\mapsto \R$,
$$\sum\limits_{x\in \mathcal{B}(R)} (f(x)-\bar{f}_R)^2\leq C'R^2\sum\limits_{x,y\in B(2R)}c(x,y)(f(x)-f(y))^2 $$
where $\bar{f}_R=\frac{1}{V(R)}\sum\limits_{y\in \mathcal{B}(R)} f(y)$.
\end{itemize}
Verifying condition $(VD)$ is straightforward in our case. Moreover Condition $(WPI(2))$ is a consequence of Theorem 4.1 in \cite{saloff-coste_1995} applied with $p=2$.
\end{proof}
\subsection{A warm-up about point processes} \label{warmup}
For reader's convenience, we recall a few facts about point processes. One can refer to \cite{Kallenberg_modern} and \cite{MKM} for more details.
A point measure $m$ is an integer-valued measure on $\R^d$ such that $m$ is finite on all compact sets of $\R^d$. We denote by $\textbf{N}$, the set of point measures. We can equip $\textbf{N}$ with a $\sigma$-field $\mathfrak{N}$ which is the smallest $\sigma$-field that makes measurable the applications $m\mapsto \int f dm$ for every $f$ in $\mathcal{F}_c(\R^d)$, the set of compactly supported continuous functions on $\R^d$.
A point process $\Theta$ is then defined as a random variable from a probability space $(\Omega,\mathcal{F},\P)$ into $(\textbf{N},\mathfrak{N})$. It is a random collection of positions which is locally finite. Sometimes, we use the following abuse of notation: we write $x\in \Theta$ as $x$ runs through the support of $\Theta$ with multiplicity. For example, if $\Theta(\{x\})=3$, $x$ is counted three times.
A natural question is to know how to characterize the distribution of a point process. Let us state the following useful criterion.
\begin{lem}[Corollary 2.3 in \cite{Kallenberg_modern}]\label{egaloi}
Let $\xi_1$ and $\xi_2$ be two point processes. Then the following are equivalent:
\begin{enumerate}
\item $\xi_1$and $\xi_2$ have the same distribution.
\item For every $f\in \mathcal{F}_c^+(\R^d)$, the set of non-negative compactly supported continuous functions, $$\E\left[\exp\left(-\int f(x) d\xi_1(x)\right)\right]=\E\left[\exp\left(-\int f(x) d\xi_2(x)\right)\right].$$
\end{enumerate}
\end{lem}

Further, a sequence of point processes $(\xi_n)_{n\in\N}$ is said to converge in law in the vague topology toward a point process $\xi$ if and only if $\int f d\xi_n\xrightarrow[n\rightarrow +\infty]{law}\int f d\xi$ for every $f\in \mathcal{F}_c(\R^d)$. One can also say that $\xi_n$ converges vaguely toward $\xi$ .
Actually, the following criterion characterizes convergence in law in the vague topology.
\begin{lem}[Theorem 4.11 in \cite{Kallenberg_modern} ]
A sequence of point processes $(\xi_n)_{n\in\N}$ converges in law toward a point process $\xi$ in the vague topology if and only if $$\displaystyle\E\left[\exp\left(-\int f(x) d\xi_n(x)\right) \right]\xrightarrow[n\rightarrow +\infty]{}\E\left[\exp\left(-\int f(x) d\xi(x)\right) \right]$$ for every $f\in \mathcal{F}_c^+(\R^d)$.
\end{lem}
\noindent We also need a criterion for the existence of a limit point process:
\begin{lem} \label{existence}
Let $(\xi_n)_{n\in\N}$ be a sequence of point processes. Let us assume that for every $f\in\mathcal{F}_c(\R^d)$ and for every $\eta\in\R$, $$\E\left[\exp\left( i\eta \int f(x) d\xi_n(x)\right) \right]\xrightarrow[n\rightarrow +\infty]{}\Phi_{f}(\eta)$$ and that $\Phi_f$ is a continuous function, then $(\xi_n)_{n\in\N}$ converges in law in the vague topology toward some point process $\xi$.
\end{lem}
\begin{proof}
We combine Corollary 4.14 in \cite{Kallenberg_modern} and strong Lévy's continuity Theorem. (See \cite{Fristedt_Gray}.)
\end{proof}

\subsection{Spine method and change of measure} \label{measure}
A key ingredient in our proof is the following spine method which was developped largely in the study of branching processes. (See \cite{LPP} or section 2 in \cite{Hushi}.) Let us begin with some useful notation. First, we introduce $\mathcal{U}:=\{\emptyset\}\cup\bigcup_{k=1}^{+\infty}\left(\N^*\right)^k$ which is called Neveu's space. Every $u=i_1\cdots i_r$ represents the labelling of a particle at generation $r$. Indeed $u$ is the $i_r$-th offspring of the $i_{r-1}$-th offspring of $\cdots$ of the $i_1$-th offspring of the root $\emptyset$. The length $r$ of $u$ is denoted by $|u|$. The parent of $u$ is denoted by $\cev{u}$. If $u,v\in\mathcal{U}$, we denote by $uv$ the concatenation of $u$ and $v$. Moreover, we introduce the partial order $\leq$ on $\mathcal{U}$ by $u\leq v$ if $u$ is an ancestor of $v$. Then let us define $\mathcal{V}:=\{(u,s_u):u\in \mathcal{U}, s_u\in\R \}$. If $\mathcal{T}$ is a subset of $\mathcal{V}$, let us define $\mathcal{T}_{gen}:=\{u\in\mathcal{U}:\exists s_u\in\R, (u,s_u)\in\mathcal{T}\}$. Finally we define the set of marked trees $$\mathcal{E}=\{\mathcal{T}\subset \mathcal{V}:\emptyset\in\mathcal{T}_{gen}, \text{ the connected component of }\emptyset\text{ in }\mathcal{T}_{gen} \text{ is a tree}\}. $$
For every $n\in\N$, we define the $\sigma$-field $\mathcal{F}_n$ on $\mathcal{E}$ by $\mathcal{F}_n:=\sigma\left(\{(u, s_u):|u|\leq n\}\right)$ and we denote by $\mathcal{F}_{\infty}$ the $\sigma$-field $\sigma\left(\bigcup_{n\in\N} \mathcal{F}_n \right)$. Let $\mathbf{T}: \mathcal{E}\mapsto \mathcal{E}$ be the identity map. If we have a probability measure $\mathbf{M}$ on $\mathcal{E}$, $\mathbf{T}$ can be seen as a random object of $\mathcal{E}$ with distribution $\mathbf{M}$. We use the notation $\E_{\mathbf{M}}$ to mean that we integrate functionnals of $\mathbf{T}$ with respect to $\mathbf{M}$. For every $\mathcal{T}:=\{(u,s_u):u\in \mathcal{U}, s_u\in\R \}\in\mathcal{E}$ with underlying tree $\mathcal{T}_{gen}$, we define a new element $\mathcal{T}^u$ of $\mathcal{E}$ starting from $u$ and its underlying tree as
$$\mathcal{T}^u_{gen}:=\{v\in\mathcal{U}: uv\in \mathcal{T}_{gen}\}$$
and
$$\mathcal{T}^u=\{(v, s_{uv}-s_u): v\in \mathcal{T}^u_{gen}\}.$$ Moreover, for every $u\in\mathcal{U}$, we define the measurable function $\mathbf{T}^u$ from $\mathcal{E}$ into $\mathcal{E}$ by $\mathbf{T}^u(\mathcal{T})=\mathcal{T}^u$ for every $\mathcal{T}\in\mathcal{E}$. All random variables $S_u$, $Z_n$, $Z_n(A)$ which are introduced in subsection \ref{debut} can be seen as measurable functions from $\mathcal{E}$ into $\R$.
Let us introduce the measure $\mathbf{P}$ on $\mathcal{E}$ which is the probability distribution of the critical branching process starting from an ancestor located at $0$ defined in subsection \ref{debut}. Then, we introduce a new measure $\mathbf{Q}$ thanks to the following change of measure:
$$\mathbf{Q}|_{\mathcal{F}_n}:=Z_n\cdot \mathbf{P}|_{\mathcal{F}_n}$$
where $Z_n$ is the number of particles at generation $n$. One can remark that $(Z_n)_{n\in\N}$ is a martingale with respect to $\left(\mathcal{F}_n\right)_{n\in\N}$ under $\mathbf{P}$.

Now, let us introduce a "size-biased" branching process. We define a new size-biased law $\nu$ on $\N$ by $\nu(k)=k\mu(k)$ for every $k\in\N$. This is a probability measure because $\mu$ has mean 1. Motions still have distribution $\mathcal{P}$. Then, we proceed recursively to construct the size-biased version of our branching process:

We start with one particle called $w_0$ located at $0$ in $\R^d$. It defines the generation 0 of the process. Now let $\mathcal{Z}_n$ be the set of particles in generation $n$. We also have a marked particle $w_n$ among $\mathcal{Z}_n$. The particle $w_n$ gives birth to $ \hat{N}_{w_n}$ children with $ \hat{N}_{w_n}\sim \nu$. The children of $w_n$ jump independently from the position of $w_n$ according to the motion law $\mathcal{P}$. Among the children of $w_n$, we choose uniformly at random a special particle called $w_{n+1}$. Moreover, every particle $u\in\mathcal{Z}_n\backslash\{w_n\}$ gives birth independently to $N_u$ offsprings where $N_u$ has law $\mu$ as in the classical branching process. The offsprings of $u$ jump from the position of $u$ according to the motion law $\mathcal{P}$. The offsprings of $\mathcal{Z}_n$ with their new positions (including the marked particle $w_{n+1}$) form the $(n+1)$-th generation of the branching process.

With this construction we get a distinguished ray of particles $(w_n)_{n\in\N}$ called the spine. Let us define the set $\mathcal{E}^*$ which consists in elements $\mathcal{T}$ of $\mathcal{E}$ with an infinite distinguished ray. The size-biased critical branching process defined above gives us a probability distribution $\mathbf{Q}^*$ on $\mathcal{E}^*$.
Then we have the following proposition establishing a link between $\mathbf{Q}$ and $\mathbf{Q}^*$. 
\begin{prop}[\cite{LPP}] \label{change}
The marginal of $\mathbf{Q}^*$ with respect to $\mathcal{E}$ (that is, we forget the distinguished ray) is $\mathbf{Q}$. Moreover, for every $n\in\N^*$ and for every particle $u$ at generation $n$,
$$\mathbf{Q}^*\left(w_n=u|\mathcal{F}_n \right)= \frac{1}{Z_n}. $$
\end{prop}
We must insist on the fact that, under $\mathbf{Q}^*$, the spine $(S_{w_n})_{n\in\N}$ is a random walk starting from $0$ with motion law $\mathcal{P}$. We denote by $\mathcal{S}$ the $\sigma$-field generated by the spine and $\mathcal{G}$ the $\sigma$-field generated by the spine and the number of children of particles in the spine and the positions of the brothers of the spine. Let us denote by $\mathfrak{B}$ the set of brothers of particles in the spine. It is clear from our construction that we get the following proposition.
\begin{prop} \label{condifreres}
Under $\mathbf{Q}^*$, conditionally on $\mathcal{G}$ (or $\mathcal{S}$), $\left( \mathbf{T}^u\right)_{u\in\mathfrak{B}}$ are independent and $\mathbf{T}^u$ has distribution $\mathbf{P}$ for every $u\in\mathfrak{B}$.
\end{prop}

\subsection{A critical branching random walk conditioned on survival in a given set} \label{secfunlemma}

We are now going to prove a key lemma for this article.
Let us recall some notation. In every generation $n$, $\{S_u, |u|=n\}$ is the collection of positions of all particles in generation $n$. $Z_n$ is the total number of particles in generation $n$ and for any closed ball $A$, we  denote by $Z_n(A)$ the number of particles lying in $A$ in generation $n$. Moreover the set of continuous functions with compact support is denoted by $\mathcal{F}_c(\R^d)$.
\begin{lem}[Key Lemma]\label{funlemma}
Let $A$ be a closed ball. We assume $\mathcal{H}_1$ or $\mathcal{H}_2$. Let $f\in\mathcal{F}_c(\R^d)$ whose values are in $\R_-$ or $i\R$ and such that $supp(f)\subset A$. Let $M>0$. Then, there exists a constant $I_{A,f}$ such that, uniformly in $x\in B(0,M\sqrt{n})$, as $n$ goes to infinity,
$$\begin{array}{ll}\displaystyle\E_{\mathbf{P}}\left[\textbf{1}\{Z_n(A-x)\geq 1\}\exp\left(\sum\limits_{|u|=n} f(S_u+x) \right) \right] \vspace{0.5 cm}\\
&\hspace{-7 cm}=(1+o_n(1))\times (2\pi n )^{-d/2}\det(\Sigma)^{-1/2}\exp\left( -\frac{1}{2n}\langle x,\Sigma^{-1}x\rangle\right)I_{A,f} .
\end{array}$$
\end{lem}
\begin{proof}
We do the proof only under the hypothesis $\mathcal{H}_2$. Assuming the hypothesis $\mathcal{H}_1$, the proof follows exactly the same lines but with a few notation changes.
We start by using the spine decomposition and the change of measure described in subsection \ref{measure}. This yields that for every $n\in\N^*$ and for every $x\in B(0,M\sqrt{n})$,
\begin{align}
\hspace{1 cm}\E_{\mathbf{P}}\left[\textbf{1}\{Z_n(A-x)\geq 1\}\exp\left(\sum\limits_{|u|=n} f(S_u+x) \right) \right]\nonumber\\
&\hspace{-7 cm}=\displaystyle\E_{\mathbf{P}}\left[\frac{\textbf{1}\{Z_n(A-x)\geq 1\}Z_n(A-x)}{Z_n(A-x)} \exp\left(\sum\limits_{|u|=n} f(S_u+x) \right)\right] \nonumber \\ 
&\hspace{-7 cm}=\displaystyle\E_{\mathbf{Q}}\left[\frac{Z_n(A-x)}{Z_n}\frac{\textbf{1}\{Z_n(A-x)\geq 1\}}{Z_n(A-x)}\exp\left(\sum\limits_{|u|=n} f(S_u+x) \right)\right] \nonumber\\ 
&\hspace{-7 cm}=\displaystyle\E_{\mathbf{Q}^*}\left[\frac{\textbf{1}\{S_{w_n}\in A-x \}}{Z_n(A-x)} \exp\left(\sum\limits_{|u|=n} f(S_u+x) \right)\right].\label{ega0}
\end{align}

In the second equality, we used the definition of $\mathbf{Q}$ and in the last one we used Proposition \ref{change}. Now, let us introduce some notation. For every $k\in\N$, let $B(w_{k+1})$ be the set of brothers of $w_{k+1}$, that is, the children of $w_k$ which are not $w_{k+1}$. $|B(w_{k+1})|$ stand for the number of individuals in $B(w_{k+1})$. If $\mathcal{Z}$ is a set of particles and $g$ a function, let us define:
$$\langle \mathcal{Z}, g \rangle := \sum\limits_{u\in \mathcal{Z}} g(S_u). $$
Further, for every particle $u$, we denote by $\mathcal{Z}_k^u$ the $k$-th generation of the branching process $\mathbf{T}^u$ starting from $u$. Moreover, the positions in $\mathbf{T}^u$ are shifted by $S_u$, that is, $S^u_v=S_v-S_u$ for every $v\in \mathcal{Z}_k^u$. For every set $\tilde{A}$, for every $k\geq 0$ and for every particle $u$, we define $Z^u_k(\tilde{A})$ by $$Z^u_k(\tilde{A})=\sum\limits_{v\in \mathcal{Z}^u_k}\delta_{S_v^u\in \tilde{A} }.$$
Further, for every particle $u$, we denote by $\rho_u$ the random variable $S_u-S_{\cev{u}}$. Then, by spinal decomposition we know that,
$$\begin{array}{ll}
\sum\limits_{|u|=n} f(S_u+x)= f(S_{w_n}+x)+\sum\limits_{k=0}^{n-1}\sum\limits_{u\in B(w_{k+1})} \langle \mathcal{Z}^u_{n-k-1},f(S_{w_k}+\rho_u+x+\cdot)\rangle
\end{array}. $$
In the same way,
$$ \begin{array}{ll}
Z_n(A-x)=\textbf{1}\{S_{w_n}\in A-x\}+\sum\limits_{k=0}^{n-1}\sum\limits_{u\in B(w_{k+1})} Z^u_{n-k-1}(A-x-S_{w_k}-\rho_u).
\end{array}$$
For sake of clarity, let us introduce for every $k\in\{0,\cdots,n-1\}$ and for every $z\in\R$:
$$\begin{array}{ll}
Y_{f,n-k}(z)&:=\sum\limits_{u\in B(w_{k+1})} \langle \mathcal{Z}^u_{n-k-1},f(z+\rho_u+\cdot)\rangle \\
Y_{n-k}(z)&:=\sum\limits_{u\in B(w_{k+1})} Z^u_{n-k-1}(A-z-\rho_u).
\end{array}$$
Moreover, by convention we define for every $z\in\R$:
$$Y_{f,0}(z):= f(z). $$
With this notation, (\ref{ega0}) becomes\\
\begin{align}
&\hspace{-0.5 cm}\E_{\mathbf{P}}\left[\textbf{1}\{Z_n( A-x)\geq 1\}\exp\left(\sum\limits_{|u|=n} f(S_u+x) \right) \right]\nonumber \\
&=\E_{\mathbf{Q}^*}\left[\frac{\textbf{1}\{S_{w_n}\in A-x\}}{1+\sum\limits_{k=0}^{n-1}Y_{n-k}(S_{w_k}+x)}\exp\left(\sum\limits_{k=0}^{n}Y_{f,n-k}(S_{w_k} +x)\right) \right] \nonumber\\
&=\E_{\mathbf{Q}^*}\left[\frac{\textbf{1}\{S_{w_n}\in A-x\}}{1+\sum\limits_{k=1}^{n}Y_{k}(S_{w_{n-k}} +x)}\exp\left(\sum\limits_{k=0}^{n}Y_{f,k}(S_{w_{n-k}}+x)\right) \right]. \label{ega1}
\end{align}
Let $K>0$. We write $\E_{\mathbf{Q}^*}\left[\cdot| \mathcal{S} \right]$ when we mean that we condition on the spine. By conditional Markov inequality we deduce that\\
\begin{align}
\mathbf{Q}^*\left(S_{w_n}\in A-x,\sum\limits_{k=K+1}^{n}Y_{k}(S_{w_{n-k}} +x)\geq 1\right)& \nonumber\\
&\hspace{-7 cm}\leq\E_{\mathbf{Q}^*}\left[\textbf{1}\{S_{w_n}\in A-x\}\E_{\mathbf{Q}^*}\left[\sum\limits_{k=K+1}^{n}Y_{k}(S_{w_{n-k}}+x) \Bigg| \mathcal{S}\right] \right].\label{ega1bis}
\end{align}
However by definition of $\mathbf{Q}^*$ and $Y_k(\cdot)$ and by Proposition \ref{condifreres}, for any $k\in\{K+1,\cdots, n\}$,
\begin{align}
\E_{\mathbf{Q}^*}\left[Y_{k}(S_{w_{n-k}}+x) \Bigg| \mathcal{S}\right]
&=\E_{\mathbf{Q}^*}\left[|B(w_ {n-k+1})|\Bigg| \mathcal{S}\right]\times \E_{\mathbf{P}} \left[Z_{k-1}(A-x-y-\rho) \right]|_{y=S_{w_{n-k}}}\nonumber\\
&=\sigma^2\times \E_{\mathbf{P}} \left[Z_{k-1}(A-x-y-\rho) \right]|_{y=S_{w_{n-k}}}\nonumber\\
&=\sigma^2\times \P\left(\hat{S}_{k-1}\in A-x-y-\rho\right)|_{y=S_{w_{n-k}}} \label{ega1tri}
\end{align}
where $(\hat{S}_k)_{k\in\N}$ is a random walk whose i.i.d increments have distribution $\mathcal{P}$ and $\rho$ is a random variable with distribution $\mathcal{P}$ which is independent of $Z_{k-1}(\cdot)$ under $\mathbf{P}$ and independent of $\hat{S}$ under $\P$. In the second equality we used the fact that $1+|B(w_{n-k+1})|$ has distribution $\nu$ under $\mathbf{Q}^*$ and $1+\sigma^2$ is the mean of $\nu$. To obtain the third equality, we used the fact that the branching process is critical under $\mathbf{P}$. Consequently, thanks to Corollary \ref{inegastand} and identity (\ref{ega1tri}), we obtain for any $k\in\{K+1,\cdots, n\}$,

\begin{align}
\E_{\mathbf{Q}^*}\left[Y_{k}(S_{w_{n-k}}+x) \Bigg| \mathcal{S}\right]&\leq \sigma^2|A|c_dk^{-d/2}.\label{ega1qua}
\end{align}
Therefore, combining identities (\ref{ega1qua}) and (\ref{ega1bis}) and Corollary \ref{inegastand} again, there exists a constant $C>0$ such that
\begin{align}
\mathbf{Q}^*\left(S_{w_n}\in A-x,\sum\limits_{k=K+1}^{n}Y_{k}(S_{w_{n-k}} +x)\geq 1\right)\vspace{0.2 cm}
&\leq \sigma^2|A|^2c_d^2\times n^{-d/2} \sum\limits_{k=K+1}^{+\infty}k^{-d/2}\nonumber\\
&\leq C\times n^{-d/2} K^{1-d/2}.\nonumber
\end{align} 
 We remark that $K^{1-d/2}=o_K(1)$  if and only if $d\geq3$. Thus, together with (\ref{ega1}), it yields\\
\begin{align}
&\hspace{-0.5 cm}\E_{\mathbf{P}}\left[\textbf{1}\{Z_n(A-x)\geq 1\}\exp\left(\sum\limits_{|u|=n} f(S_u+x) \right) \right]\nonumber \\
&=\E_{\mathbf{Q}^*}\left[\frac{\textbf{1}\{S_{w_n}\in(A-x)\}}{1+\sum\limits_{k=1}^{K}Y_{k}(S_{w_{n-k}} +x)}\exp\left(\sum\limits_{k=0}^{K}Y_{f,k}(S_{w_{n-k}}+x)\right) \right]+o_K(1)n^{-d/2}. \label{ega2}
\end{align}
Now let us introduce some new notation. For every $k\in\N^*$ and for every $z\in\R$,
$$\begin{array}{ll}
\tilde{Y}_{f,k}(z)&:=\sum\limits_{u\in B(w_{k+1})} \langle \mathcal{Z}^u_{k-1},f(z+\rho_u+\cdot)\rangle \\
\tilde{Y}_k(z)&:=\sum\limits_{u\in B(w_{k+1})} Z^u_{k-1}(A-z-\rho_u).
\end{array}$$
Moreover, by convention, we define for every $z\in\R$:
$$\tilde{Y}_{f,0}(z):= f(z). $$
With this new notation, by reversing time in the spine, (\ref{ega2}) yields,

\begin{align}
\E_{\mathbf{P}}\left[\textbf{1}\{Z_n( A-x)\geq 1\}\exp\left(\sum\limits_{|u|=n} f(S_u+x) \right) \right]&\nonumber \\
&\hspace{-7,2 cm}=\E_{\mathbf{Q}^*}\left[\frac{\textbf{1}\{S_{w_n}-S_{w_K}\in(A-x-S_{w_K})\}}{1+\sum\limits_{k=1}^{K}\tilde{Y}_{k}((S_{w_{n}}-S_{w_K})+x+(S_{w_K}-S_{w_k}))}\exp\left(\sum\limits_{k=0}^{K}\tilde{Y}_{f,k}((S_{w_{n}}-S_{w_k})+x)\right) \right]\nonumber\\
&\hspace{-6.8 cm}+o_K(1)n^{-d/2}. \label{ega3}
\end{align}
For every $z\in\mathbb{Z}^d$, let us denote $p_{n-K}(z)=\mathbf{Q}^*(S_{w_n}-S_{w_K}=z)$.
Let us define also the function 
 $$\begin{array}{ccccl}
F_{K,A} & : & \R^{K+1} & \longrightarrow & \C \\
 & & (y,y_1,\cdots,y_K) & \mapsto & \E_{\mathbf{Q}^*}\left[\frac{1}{1+\sum\limits_{k=1}^{K}\tilde{Y}_{k}(y-y_k)} \exp\left(\sum\limits_{k=0}^{K}\tilde{Y}_{f,k}(y-y_k)\right)\right].\\
\end{array}$$
Therefore, by (\ref{ega3}) and the independence under $\mathbf{Q}^*$ between $S_{w_n}-S_{w_K}$, $(S_{w_k})_{1\leq k\leq K}$ and $\left(\tilde{Y}_{f,k}(\cdot)\right)_{1\leq k\leq K}$, it holds that\\
\begin{align}\label{ega4}
&\hspace{-0.5 cm}\E_{\mathbf{P}}\left[\textbf{1}\{Z_n( A-x)\geq 1\}\exp\left(\sum\limits_{|u|=n} f(S_u+x) \right) \right]\nonumber \\
&=\E_{\mathbf{Q}^*}\left[\sum\limits_{y\in A\cap\mathbb{Z}^d}p_{n-K}(y-x-S_{w_K})F_{K,A}\left(y,S_{w_k},1\leq k\leq K\right)\right]+o_K(1)n^{-d/2}.
\end{align}
Using Theorem \ref{loclimitlat}, it holds that, uniformly in $x\in B(0,M\sqrt{n})$,

\begin{align}\label{ega5}
\E_{\mathbf{P}}\left[\textbf{1}\{Z_n(A-x)\geq 1\}\exp\left(\sum\limits_{|u|=n} f(S_u+x) \right) \right]\nonumber 
\\ &\hspace{-7.2 cm}=\hspace{0.2 cm}(2\pi(n-K))^{-d/2}\det(\Sigma)^{-1/2}\nonumber  \\
&\hspace{-6.8 cm}\times\sum\limits_{y\in A\cap\mathbb{Z}^d }\E_{\mathbf{Q}^*}\left[ \exp\left( -\frac{1}{2(n-K)}\langle V_{x,y,K},\Sigma^{-1}V_{x,y,K}\rangle\right)F_{K,A}\left(y,S_{w_k},1\leq k\leq K\right)\right]\nonumber\\
&\hspace{-6.8cm}+o_K(1)n^{-d/2}+o_{n,K}(1)n^{-d/2}
\end{align}
where for every $K\in\N$, $o_{n,K}(1)$ tends to zero as $n$ goes to infinity and $V_{x,y,K}=y-x-S_{w_K}$.
However, we remark that,
\begin{align}
\Bigg|\frac{1}{2(n-K)}\langle V_{x,y,K},\Sigma^{-1}V_{x,y,K}\rangle-\frac{1}{2n}\langle x,\Sigma^{-1}x\rangle\Bigg|\nonumber\\
&\hspace{-7 cm}\leq\left( \frac{1}{2(n-K)}-\frac{1}{2n}\right)| \langle x, \Sigma^{-1} x\rangle|+ \frac{1}{2(n-K)}\Bigg|\langle V_{x,y,K},\Sigma^{-1}V_{x,y,K}\rangle-\langle x,\Sigma^{-1} x\rangle \Bigg|\nonumber\\
&\hspace{-7 cm}\leq \frac{||\Sigma^{-1}||}{2(n-K)}\left(\frac{K||x||^2}{n}+||y||^2+||S_{w_K}||^2+2||x||(||y||+||S_{w_K}||)+2||y||\cdot||S_{w_K}|| \right)\nonumber\\
&\hspace{-7 cm}\leq \frac{||\Sigma^{-1}||}{2(n-K)}\left(\frac{K||x||^2}{n}+2||y||^2+2||S_{w_K}||^2+2||x||(||y||+||S_{w_K}||) \right).\label{ega5bisbis}
\end{align}
Therefore, on the event $\{ ||S_{w_K}||\leq K\}$, we can use (\ref{ega5bisbis}) and the inequality $|1-e^{-t}|\leq 2|t|$ in a neighborhood of 0 to show that there is $N_K\in\N^*$ such that for every $n\geq N_K$  and for every $x\in B(0,M\sqrt{n})$,
\begin{align}
\Bigg|\exp\left( -\frac{1}{2(n-K)}\langle V_{x,y,K},\Sigma^{-1}V_{x,y,K}\rangle\right)-\exp\left( -\frac{1}{2n}\langle x,\Sigma^{-1}x)\rangle\right)\Bigg|& \nonumber\\
&\hspace{-10 cm}\leq \exp\left( -\frac{1}{2n}\langle x,\Sigma^{-1}x\rangle\right)\times \frac{2||\Sigma^{-1}||}{(n-K)}\left(\frac{K||x||^2}{n}+||y||^2+||S_{w_K}||^2+||x||(||y||+||S_{w_K}||) \right)\nonumber\\
&\hspace{-10 cm}\leq \frac{2||\Sigma^{-1}||}{(n-K)} \left(\frac{KM^2}{2}+C_A^2+K^2 +M\sqrt{n}(C_A+K)\right) \label{ega5tri}
\end{align}
where $C_A=\underset{y\in A}\max ||y||$.
Furthermore, by Proposition \ref{heat},
\begin{align}
\mathbf{Q}^*(||S_{w_K}||\geq K)=o_K(1). \label{ega5bis}
\end{align}
Finally, using inequalities (\ref{ega5tri}) and (\ref{ega5bis})  in identity (\ref{ega5}), we know that uniformly in $x\in B(0,M\sqrt{n})$,\\
\begin{align}
&\E_{\mathbf{P}}\left[\textbf{1}\{Z_n( A-x)\geq1\}\exp\left(\sum\limits_{|u|=n} f(S_u+x) \right) \right]\nonumber \\
&\hspace{0.5 cm}=(2\pi n)^{-d/2}\det(\Sigma)^{-1/2}\exp\left( -\frac{1}{2n}\langle x,\Sigma^{-1}x)\rangle\right) \sum\limits_{y\in A\cap\mathbb{Z}^d}\E_{\mathbf{Q}^*}\left[ F_{K,A}\left(y,S_{w_k},1\leq k\leq K\right)\right] \nonumber\\
& \hspace{0.9 cm}+o_K(1)n^{-d/2}+o_{n,K}(1)n^{-d/2}.\label{ega6}
\end{align}
Now, let us introduce the function $G_{A,f}$ defined by

$$\begin{array}{ccccl}
G_{A,f} & : & \R & \to & \C \\
 & & y & \mapsto &  \displaystyle\E_{\mathbf{Q}^*}\left[\frac{1}{1+\sum\limits_{k=1}^{+ \infty}\tilde{Y}_{k}(y-S_{w_k})} \exp\left(\sum\limits_{k=0}^{+\infty}\tilde{Y}_{f,k}(y-S_{w_k})\right)\right]\\
\end{array}.$$
We define $I_{A,f}:=\sum\limits_{y\in A\cap\mathbb{Z}^d} G_{A,f}(y)$. $G_{A,f}$ is well-defined because the infinite sums in its expression are actually finite sums almost surely. Indeed, we can prove exactly as in (\ref{ega1qua}) that for every $k\in\N^*$,
\begin{align}
\mathbf{Q}^*\left(\tilde{Y}_{k}(y-S_{w_k})\geq 1\right)&\leq \E_{\mathbf{Q}^*}\left[\tilde{Y}_{k}(y-S_{w_k}) \right]\leq \sigma^2|A|c_dk^{-d/2}.\nonumber
\end{align}
As $d\geq 3$, this is summable. Thus, by Borel-Cantelli Lemma, almost surely, there is only a finite number of integers $k\in\N^*$ such that $\tilde{Y}_{k}(y-S_{w_k})\geq 1$. As $supp(f)\subset A$, there is also only a finite number of integers $k\in\N^*$ such that $\tilde{Y}_{f,k}(y-S_{w_k})\geq 1$. Consequently, the random  variable inside the expectation in the definition of $G_{A,f}$ is well-defined. Moreover the modulus of this quantity is lower than 1. Therefore, $G_{A,f}$ is well-defined.

Besides, by the dominated convergence theorem, $$\displaystyle\sum\limits_{y\in A\cap\mathbb{Z}^d}\E_{\mathbf{Q}^*}\left[ F_{K,A}\left(y,S_{w_k},1\leq k\leq K\right)\right]=I_{A,f}+o_K(1).$$
Finally, using this in (\ref{ega6}), we proved that uniformly in $x\in B(0,M\sqrt{n})$,\\
\begin{align}
\E_{\mathbf{P}}\left[\textbf{1}\{Z_n( A-x)\geq 1\}\exp\left(\sum\limits_{|u|=n} f(S_u+x) \right) \right]&\nonumber \\
&\hspace{-7 cm}=(2\pi n)^{-d/2}\det(\Sigma)^{-1/2}\exp\left( -\frac{1}{2n}\langle x,\Sigma^{-1}x)\rangle\right) \times I_{A,f}\nonumber\\
&\hspace{-6.6 cm}+o_K(1)n^{-d/2}+o_{n,K}(1)n^{-d/2}.\label{ega7}
\end{align}
It concludes the proof.

\end{proof}
In the proof of Lemma \ref{funlemma}, the quantity $I_{A,f}$ naturally arises. It is a very important quantity in the sequel of this article. We often refer to its definition.
\begin{defi}\label{defiprat}
Let $A$ be a closed ball and let $f$ be a continuous function whose support is in $A$ and whose values are in $-\R$ or in $i\R$.
With the notation which is introduced in the proof of Lemma \ref{funlemma}, we recall that:
$$\begin{array}{ccccl}
G_{A,f} & : & \R & \to & \C \\
 & & y & \mapsto &  \displaystyle\E_{\mathbf{Q}^*}\left[\frac{1}{1+\sum\limits_{k=1}^{+ \infty}\tilde{Y}_{k}(y-S_{w_k})} \exp\left(\sum\limits_{k=0}^{+\infty}\tilde{Y}_{f,k}(y-S_{w_k})\right)\right].\\
\end{array}$$
\begin{itemize}
\item Under the hypothesis $\mathcal{H}_1$, we define $I_{A,f}:=\int_A G_{A,f}(y)dy$.
\item Under the hypothesis $\mathcal{H}_2$, we define $I_{A,f}:=\sum\limits_{y\in A\cap\mathbb{Z}^d} G_{A,f}(y)$.
\end{itemize}
Moreover, we often use the notation $I_A:=I_{A,0}$.
\end{defi}
We remark that the definition \ref{defiprat} is reminiscent of the structure of backward tree introduced by Kallenberg in \cite{kallenberg_article}.

As a particular case of the Lemma \ref{funlemma}, we can estimate the probability for the branching process to survive in a specified area of $\R^d$.
\begin{prop} \label{equiprob}
Let $A$ be a closed ball. We assume hypotheses $\mathcal{H}_1$ or $\mathcal{H}_2$. For every $n\in\N^*$, let us define $a_n=a\sqrt{n}$ for some $a\in\R^d$. Then we have the following equivalent\\
$$\mathbf{P}(Z_n(A-a_n)\geq 1) \underset{n\rightarrow +\infty}\sim I_A\times (2\pi)^{-d/2}\det(\Sigma)^{-1/2}e^{-\langle a,\Sigma^{-1} a\rangle/2} n^{-d/2}$$
with $I_A$ defined in definition \ref{defiprat}.
\end{prop}
\begin{proof}
Apply Lemma \ref{funlemma} with $f\equiv0$ and $M$ larger than $||a||$.
\end{proof}
We are now able to prove our first proposition concerning convergence of point processes.
\begin{prop} \label{limitpp}
Let $A$ be a closed ball. We assume hypotheses $\mathcal{H}_1$ or $\mathcal{H}_2$. For every $n\in\N$, let us define $a_n=a\sqrt{n}$ for some $a\in\R^d$.
Then, under probability measure $\mathbf{P}$, we have the following convergence in law\\
$$\mathcal{L}\left(\sum\limits_{\substack{|u|=n}}\delta_{S_u+a_n}|Z_n(A-a_n)\geq 1\right)\xrightarrow[n\rightarrow+\infty]{law}\mathcal{N}_A$$ where $\mathcal{N}_A$ is a point process on $A$ which does not depend on $a$.
Moreover for every $f\in\mathcal{F}_c^+(\R^d)$ such that $supp(f)\subset A$,
$$\E\left[\exp\left(-\int f(x) \mathcal{N}_A(dx)\right) \right]=I_A^{-1}I_{A,-f}$$
where the terms of the right-hand side were defined in definition \ref{defiprat}.
\end{prop}

\begin{proof}
By Lemma \ref{existence}, it is enough to prove that for every $f\in\mathcal{F}_c(\R^d)$ such that $supp(f)\subset A$, there exists a continuous function $\Phi_f$ such that for every $\eta\in\R$,
$$\E_{\mathbf{P}}\left[\exp\left(i\eta\sum\limits_{\substack{|u|=n}}f(S_u+a_n)\right)\Bigg|Z_n(A-a_n)\geq 1\right]\xrightarrow[n\rightarrow +\infty]{}\Phi_f(\eta).$$\\
By Lemma \ref{funlemma}, 
$$\E_{\mathbf{P}}\left(\exp\left(i\eta\sum\limits_{\substack{|u|=n}}f(S_u+a_n) \right)\textbf{1}\{Z_n(A-a_n)\geq 1 \} \right)$$
is asymptotically
 $$(1+o_n(1))\frac{I_{A,i\eta f}}{(2\pi)^d}\det(\Sigma)^{-1/2}e^{-\langle a,\Sigma^{-1 }a\rangle/2}\frac{1}{n^{d/2}}.$$
Combining this with Proposition \ref{equiprob} we get that 
$$\E_{\mathbf{P}}\left(\exp\left(i\eta\sum\limits_{\substack{|u|=n}}f(S_u) \right)\Bigg|Z_n(A-a_n)\geq 1\right)\xrightarrow[n\rightarrow+\infty]{}I_A^{-1}I_{A,i\eta f}.$$

Looking at the expression of $I_{A,i\eta f}$ which is given in definition \ref{defiprat}, we deduce from dominated convergence theorem that $\eta\mapsto I_{A,i\eta f}$  is a continuous function. It concludes the proof of the first part of Proposition \ref{limitpp} concerning convergence in law toward a point process $\mathcal{N}_A$. The second part of Proposition \ref{limitpp} concerning the Laplace transform of $\mathcal{N}_A$ is obtained by exactly the same computations. We just have to replace $i\eta f$ by $-f$.
\end{proof}
\section{Proof of Theorem \ref{cv}}  \label{seccv} 

\paragraph{Strategy of the proof of Theorem \ref{cv}}~

If a critical branching process starts from a single particle and if we condition it to visit a closed ball $A$, we proved in the previous section that the limiting point process is $\mathcal{N}_A$. However, most of the critical branching processes starting from particles in $\Lambda_0^{d,\theta}$ will not reach $A$ because of transience (in dimension $d\geq 3$) or because of extinction of the branching process. Let us make this intuition more quantitative. By homogeneity of the Poisson point process $\Lambda_0^{d,\theta}$, for every $M\geq 1$, $$\Lambda_0^{d,\theta}(B(0,M\sqrt{n}))\simeq \Theta(n^{d/2}).$$
Moreover, any particle of $\Lambda_0^{d,\theta}$ located at $x\in\R^d$ with $||x||\gg \sqrt{n}$ is too far from $A$ to have descendants in $A$ at time $n$. Indeed, a centered random walk with second moment is at distance $O\left(\sqrt{n}\right)$ from 0 at time $n$. Moreover, we proved in the previous section that a particle in $\Lambda_0^{d,\theta}$ located at $x\in B(0,M\sqrt{n})$ has descendants which reach $A$ at time $n$ with probability $$\mathbf{P}(Z_n(A-x)\geq 1)\simeq\Theta(n^{-d/2}).$$
If $Z_n^{(y)}(\cdot)$ is the occupation measure of the $n$-th generation of the critical branching process starting from a particle $y$ in the support of $\Lambda_0^{d,\theta}$, combining both previous approximations, we get that
$$\int \textbf{1}\{Z_n^{(y)}(A)\geq 1\}d\Lambda_0^{d,\theta}(y)\simeq \Theta(n^{d/2})\times\Theta(n^{-d/2})=\Theta(1). $$
Therefore, the number $P_A$ of particles in $\Lambda_0^{d,\theta}$ whose descendants reach $A$ is of order $\Theta(1)$. Moreover, for each of these particles, the positions of the descendants in $A$ form a point process distributed as $\mathcal{N}_A$. That is why, we will get an independent layering of $P_A$ copies of $\mathcal{N}_A$.

Now, let us prove the following lemma.
\begin{lem} \label{intensity}
Let $A$ be a closed ball. Let $f\in\mathcal{F}_c(\R^d)$ whose values are in $\R_-$ or in $i\R$ and such that $supp(f)\subset A$. Then, assuming hypothesis $\mathcal{H}_1$,
$$\displaystyle\E_{\mathbf{P}}\left[\int_{\R^d}\textbf{1}\{Z_n(A-x)\geq 1\}\exp\left( \sum\limits_{\substack{|u|=n}}f(S_u+x)\right)dx \right]\underset{n\rightarrow +\infty}\sim I_{A,f}$$
where the definition of $I_{A,f}$ is given in definition \ref{defiprat}.
Moreover, assuming hypothesis $\mathcal{H}_2$,
$$\displaystyle\E_{\mathbf{P}}\left[\sum\limits_{x\in\mathbb{Z}^d}\textbf{1}\{Z_n(A-x)\geq 1\}\exp\left( \sum\limits_{\substack{|u|=n}}f(S_u+x)\right) \right]\underset{n\rightarrow +\infty}\sim I_{A,f}.$$
\end{lem}
\begin{proof}
Let us prove this lemma under the hypothesis $\mathcal{H}_2$. Assuming hypothesis $\mathcal{H}_1$, the proof is similar.
For sake of clarity, we do the proof only with $f\equiv 0$. Let $M>0$. Let us use the notation $B_{\mathbb{Z}^d}(0,t)$ for the set of elements of $\mathbb{Z}^d$ whose euclidean norm is less than $t$. $B(0,t)$ denotes the standard euclidean ball.
By Markov inequality and criticality of the branching process under $\mathbf{P}$, we get for every $n\in\N$,
\begin{align}
\sum\limits_{x\in B_{\mathbb{Z}^d}(0, M\sqrt{n})^c}\mathbf{P}(Z_n(A-x)\geq 1)&\leq \sum\limits_{x\in B_{\mathbb{Z}^d}(0, M\sqrt{n})^c}\E_{\mathbf{P}}\left[Z_n(A-x) \right]\nonumber\\
&=\sum\limits_{x\in B_{\mathbb{Z}^d}(0, M\sqrt{n})^c}\P(\hat{S}_n\in A-x). \label{terme2}
\end{align}
where $\left(\hat{S}_k \right)_{k\in\N}$ is a branching random walk with motion law $\mathcal{P}$. Therefore, by Proposition \ref{heat}, there exists $C>0$ such that for every $n\in\N$:
\begin{align}
\sum\limits_{x\in B_{\mathbb{Z}^d}(0, M\sqrt{n})^c}\mathbf{P}(Z_n(A-x)\geq 1)&\leq Cn^{-d/2}\sum\limits_{x\in B_{\mathbb{Z}^d}(0, M\sqrt{n})^c}\exp\left(-\frac{||x||^2}{C n}\right). \label{terme22}
\end{align}
Thus, by standard inequalities, there exists a constant $C'>0$ such that for every $n\in\N$,
\begin{align}
\sum\limits_{x\in B_{\mathbb{Z}^d}(0, M\sqrt{n})^c}\mathbf{P}(Z_n(A-x)\geq 1)&\leq C'n^{-d/2}\int_{B(0,M\sqrt{n})^c} \exp\left(-\frac{||x||^2}{C'n}\right) dx\nonumber\\
&= C'\int_{B(0,M)^c} \exp\left(-\frac{||x||^2}{C'}\right) dx\nonumber\\
&=o_M(1).\label{terme23}
\end{align}
Moreover, by Lemma \ref{funlemma},
\begin{align}
\sum\limits_{x\in B_{\mathbb{Z}^d}(0, M\sqrt{n})}\mathbf{P}(Z_n(A-x)\geq 1)&\nonumber\\
&\hspace{-4 cm} =(2\pi n )^{-d/2}\det(\Sigma)^{-1/2}I_{A}
\times\sum\limits_{x\in B_{\mathbb{Z}^d}(0, M\sqrt{n})}\exp\left( -\frac{1}{2n}\langle x,\Sigma^{-1}x\rangle\right)+M^do_n(1).\label{terme24}
\end{align}
However, one can observe that,
\begin{align}
\hspace{-1 cm}(2\pi n )^{-d/2}\det(\Sigma)^{-1/2}\times\sum\limits_{x\in B_{\mathbb{Z}^d}(0, M\sqrt{n})}\exp\left( -\frac{1}{2n}\langle x,\Sigma^{-1}x\rangle\right)\nonumber\\
&\hspace{-9 cm}=(2\pi )^{-d/2}\det(\Sigma)^{-1/2}\int_{B(0,M)}\exp\left( -\frac{1}{2}\langle x,\Sigma^{-1}x\rangle\right)dx\nonumber+M^do_n(1)\\
&\hspace{-9 cm}=1+o_M(1)+M^do_n(1).\nonumber
\end{align}
By substituting this into identity (\ref{terme24}), we get\\
\begin{align}
\sum\limits_{x\in B_{\mathbb{Z}^d}(0, M\sqrt{n})}\mathbf{P}(Z_n(A-x)\geq 1)&=I_A+o_M(1)+M^do_n(1) .\label{terme25}
\end{align}
Then, combining identities (\ref{terme25}) and (\ref{terme23}) yields
\begin{align}
\sum\limits_{x\in \mathbb{Z}^d}\mathbf{P}(Z_n(A-x)\geq 1)=I_A+o_M(1)+M^do_n(1)
\end{align}
which concludes the proof.

\end{proof}
~\\
Now we are ready to prove Theorem \ref{cv}.
\begin{proof}
We only work on the case where $X$ is a constant $\theta$. The general case is obtained by integrating the constant case with respect to the law of $X$, that is, $\P_X$, and by using the dominated convergence theorem.
First we prove Theorem \ref{cv} under the hypothesis $\mathcal{H}_2$.
We recall that under this hypothesis, $\Lambda_0^{d,\theta}$ is a  point process on $\mathbb{Z}^d$ such that each site contains a number of particles following a Poisson distribution of parameter $\theta$. Let $f\in\mathcal{F}_c(\R^d)$. There exists a closed ball $A$ such that $supp(f)\subset A$.  Let $\eta$ be a real number. Recall that, by a small abuse of notation, for any point process $\Theta$, $"x\in\Theta"$ takes into account the multiplicity of $x$ in $\Theta$.
Then, standard computations yield
\begin{align}
\E\left[\exp\left(i\eta \int f(x) \Lambda_n^{d,\theta}(dx)\right) \right]&=\displaystyle\E\left[ \prod\limits_{x\in \Lambda_0^{d,\theta}}\E_{\mathbf{P}}\left[\exp\left( \sum\limits_{\substack{|u|=n} } i\eta f(S_u+x)\right)\right] \right]\nonumber\\
&=\displaystyle\exp\left(\theta\sum_{x\in\mathbb{Z}^d}\E_{\mathbf{P}}\left[\exp\left( \sum\limits_{\substack{|u|=n}} i\eta f(S_u+x)\right) -1 \right] \right).\label{terme31}
\end{align}
Thus, we only need to investigate the asymptotical behaviour of
$$\sum_{x\in\mathbb{Z}^d}\E_{\mathbf{P}}\left[\exp\left( \sum\limits_{\substack{|u|=n}} i\eta f(S_u+x)\right) -1 \right] $$
However, it holds that
\begin{align}
\sum_{x\in\mathbb{Z}^d}\E_{\mathbf{P}}\left[\exp\left( \sum\limits_{\substack{|u|=n}} i\eta f(S_u+x)\right) -1 \right]&\nonumber\\
&\hspace{-6 cm}=\sum_{x\in\mathbb{Z}^d}\E_{\mathbf{P}}\left[\textbf{1}\{Z_n(A-x)\geq 1\}\left(\exp\left( \sum\limits_{\substack{|u|=n}} i\eta f(S_u+x)\right) -1 \right)\right]. \label{terme32}
\end{align}
By Lemma \ref{intensity}, this quantity (\ref{terme32}) converges toward $I_{A,i\eta f}-I_A$. Therefore, together with identity (\ref{terme31}), we deduce
$$
\E\left[\exp\left(i\eta \int f(x) \Lambda_n^{d,\theta}(dx)\right) \right]=(1+o_n(1))\exp\left(\theta (I_{A,i\eta f}-I_A)\right).
$$
By the description of $I_{A,i\eta f}$ given in definition \ref{defiprat}, the function $\eta\mapsto\exp\left(\theta (I_{A,i\eta f}-I_A)\right)$ is continuous. That is why, by Proposition \ref{existence}, there exists a point process  $\Lambda_{\infty}^{d,\theta}$ such that
$$\Lambda_{n}^{d,\theta} \xrightarrow[ n\rightarrow +\infty]{law} \Lambda_{\infty}^{d,\theta}. $$
In order to do the proof under the hypothesis $\mathcal{H}_1$, we just need an identity which is similar to (\ref{terme31}). This can be obtained thanks to the exponential formula for Poisson Point Process. (See \cite{Kallenberg_modern}.) If we replace sums by integrals, the rest of the proof follows the same lines.
Now, let $g\in\mathcal{F}_c(\R^d)^+$ such that $supp(g)\subset A$. Following exactly the same computations than above, we get that,
$$\begin{array}{ll}
\displaystyle \E\left[\exp\left(-\int g(x) \Lambda_{\infty}^{d,\theta}(dx) \right) \right]&= \exp\left(\theta (I_{A,-g}-I_A)\right).
\end{array}$$
Consequently, by Proposition \ref{limitpp},
$$\begin{array}{ll}
\displaystyle \E\left[\exp\left(-\int g(x) \Lambda_{\infty}^{d,\theta}(dx) \right) \right]= \exp\left(\theta I_A\left( \E\left[\exp\left(-\int g(x) \mathcal{N}_A(dx) \right)\right]-1\right)\right).
\end{array}$$
Therefore, $\Lambda_{\infty}^{d,\theta}$ is a Poissonian sum of i.i.d copies of $\mathcal{N}_A$, as stated in Theorem \ref{cv}.
\end{proof}

In the proof of Theorem \ref{cv}, we obtained the Laplace transform of $\Lambda^{d,X}_{\infty}$.\\
\begin{prop}[  Laplace transform]\label{laplace}
For every $f\in\mathcal{F}_c^+(\R^d)$ whose support is included in some closed ball $A$, we have\\
$$\displaystyle\E\left[\exp\left(-\int f(x) \Lambda_{\infty}^{d,X}(dx)\right)\right]=\E\left[\exp\left(X\times\left(\displaystyle I_{A,-f}-I_A \right) \right) \right]$$
where $I_A$ and $I_{A,-f}$ are defined in definition \ref{defiprat}. Another reformulation is
$$\displaystyle\E\left[\exp\left(-\int f(x) \Lambda_{\infty}^{d,X}(dx)\right)\right]=\E\left[\exp\left(X I_A R_{f,A}\right)\right] $$
where $$R_{f,A}=\E\left[\exp\left(-\int f(x)\mathcal{N}_A(dx) \right)\right]-1.$$\\
\end{prop}
\begin{rem}\label{compat}
Let $A$ and $B$ be two closed balls such that $A\subset B$. By the previous Proposition \ref{laplace} we get two formulas to compute the Laplace transform of $ \Lambda^{d, X}_{\infty}$, one with respect to $A$, the other one with respect to $B$. These formulas must be equal. This is discussed in the following section. 
\end{rem}

\section{Compatibility of $\Lambda^{d,X}_{\infty}$} \label{sec4}
\label{compatibility}
Theorem \ref{cv} shows that $\Lambda_{\infty}^{d,X}$, as a limit of $\left(\Lambda_n^{d,X}\right)_{n\in\N}$, is a point process. In particular, it must be compatible. Here, we give an independent proof of compatibility of the limiting point processes of the form $\Lambda_{\infty}^{d,X}$ obtained in Theorem \ref{cv}. In all this section, $X$ is a non-negative random variable.
\begin{lem}\label{quot}
Let $A_1$ and $A_2$ be two closed balls such that $A_1\subset A_2$. We assume hypotheses $\mathcal{H}_1$ or $\mathcal{H}_2$. Then\\
$$\displaystyle\P(\mathcal{N}_{A_2}(A_1)\geq 1)=\frac{I_{A_1}}{I_{A_2}}$$
where $I_{A_1}$ and $I_{A_2}$ are defined in definition \ref{defiprat}.
\end{lem}
\begin{proof}
By Proposition \ref{limitpp}, we know that under $\mathbf{P}$,
$$\mathcal{L}\left(\sum\limits_{|u|=n}\delta_{S_u}\Bigg|Z_n(A_2)\geq 1\right)\xrightarrow[n\rightarrow +\infty]{law} \mathcal{N}_{A_2}.$$ \\
Then, by Lemma 4.1 in \cite{Kallenberg_modern}, we get\\
$$\P(\mathcal{N}_{A_2}(A_1)\geq 1)=\underset{n\rightarrow +\infty}\lim  \mathbf{P}(Z_n(A_1)\geq 1|Z_n(A_2)\geq 1)=\frac{\mathbf{P}(Z_n(A_1)\geq 1)}{\mathbf{P}(Z_n(A_2)\geq 1)}$$\\
The estimate given by Proposition \ref{equiprob} concludes the proof.
\end{proof}
\begin{lem} \label{lemi}
Let $A_1$ and $A_2$ be two closed balls such that $A_1\subset A_2$. We assume hypotheses $\mathcal{H}_1$ or $\mathcal{H}_2$.
$$\mathcal{L}\left(\mathcal{N}_{A_2}(\cdot\cap A_1) \Bigg|\mathcal{N}_{A_2}(A_1)\geq 1\right)\overset{law}=\mathcal{N}_{A_1}.$$
\end{lem}
\begin{proof}
Let $f\in\mathcal{F}_c^+(\R^d)$ such that $supp(f)\subset A_1$.
Let us observe that,
\begin{align} \label{egasec1}
\displaystyle\E\left[\exp\left(- \int f(x) \mathcal{N}_{A_2}(dx) \right)\textbf{1}\{\mathcal{N}_{A_2}(A_1)\geq 1\}\right]&\nonumber\\
&\hspace{-7 cm}=\E\left[\exp\left(-\int f(x) \mathcal{N}_{A_2}(dx) \right)\right]-\P(\mathcal{N}_{A_2}(A_1)=0).
\end{align}
However by Proposition \ref{limitpp},\\
\begin{align} \label{egasec2}
\displaystyle\underset{n\rightarrow +\infty}\lim \frac{\E_{\mathbf{P}}\left[\exp\left(-\sum\limits_{|u|=n}f(S_u) \right)\textbf{1}\{Z_n(A_2)\geq 1\} \right]}{\mathbf{P}(Z_n(A_2)\geq 1)}=\E\left[\exp\left(-\int f(x) \mathcal{N}_{A_2}(dx)\right)\right].
\end{align}
The left-hand side in (\ref{egasec2}) can be rewritten as
$$\displaystyle\frac{\E_{\mathbf{P}}\left[\exp\left(-\sum\limits_{|u|=n}f(S_u) \right)\textbf{1}\{Z_n(A_1)\geq 1\} \right]+\mathbf{P}(Z_n(A_2)\geq 1,Z_n(A_1)=0)}{\mathbf{P}(Z_n(A_1)\geq 1)}\times \frac{\mathbf{P}(Z_n(A_1)\geq 1)}{\mathbf{P}(Z_n(A_2)\geq 1)} .$$
By Lemma \ref{quot} and Propositions \ref{equiprob} and \ref{limitpp}, this converges toward

$$ \left( \E\left[\exp\left(-\int f(x) \mathcal{N}_{A_1}(dx) \right)\right]+\frac{1}{\P(\mathcal{N}_{A_2}(A_1)\geq 1)}-1\right)\P(\mathcal{N}_{A_2}(A_1)\geq 1).$$
Consequently, together with (\ref{egasec2}), this implies that
\begin{align}
\E\left[\exp\left( -\int f(x) \mathcal{N}_{A_2}(dx) \right)\right]&\nonumber\\
&\hspace{-4 cm}=\left( \E\left[\exp\left(-\int f(x) \mathcal{N}_{A_1}(dx) \right)\right]+\frac{1}{\P(\mathcal{N}_{A_2}(A_1)\geq 1)}-1\right)\P(\mathcal{N}_{A_2}(A_1)\geq 1).\label{egasec3}
\end{align}
Therefore, using identities (\ref{egasec1}) and (\ref{egasec3}), we know that\\
\begin{align}
\E\left[\exp\left(-\int f(x) \mathcal{N}_{A_2}(dx) \right)\Bigg|\mathcal{N}_{A_2}(A_1)\geq 1\right]&\nonumber\\
&\hspace{-6.5cm}=\E\left[\exp\left(- \int f(x) \mathcal{N}_{A_1}(dx) \right)\right]+\frac{1}{\P(\mathcal{N}_{A_2}(A_1)\geq 1)}-1-\frac{\P(\mathcal{N}_{A_2}(A_1)=0)}{\P(\mathcal{N}_{A_2}(A_1)\geq 1)}.\nonumber
\end{align}
This is exactly $\E\left[\exp\left(- \int f(x) \mathcal{N}_{A_1}(dx)\right)\right]$. Thus, by Lemma \ref{egaloi}, this finishes the proof.

\end{proof}
Thanks to the two previous lemmas, we are now able to deduce the compatibility of $\Lambda^{d,X}_{\infty}$.
\begin{prop}
Let $A_1$ and $A_2$ be two closed balls such that $A_1\subset A_2$. We assume hypotheses $\mathcal{H}_1$ or $\mathcal{H}_2$.
In Theorem \ref{cv}, we get a Poissonian way to define $\xi_{A_1}:=\Lambda^{d,X}_{\infty}(\cdot\cap A_1)$ and $\xi_{A_2}:=\Lambda^{d,X}_{\infty}(\cdot\cap A_2)$. Then,
$$\xi_{A_2}(\cdot\cap A_1)\overset{law}=\xi_{A_1}.$$
\end{prop}
\begin{proof}
Let us first construct $\xi_{A_2}$ with the construction given in Theorem \ref{cv}.
Consider $P_{A_2}$ a Poisson Random variable of parameter $XI_{A_2}$ and $\left(\mathcal{N}_{A_2}^{(k)}\right)_{k\in\N}$ a family of i.i.d copies of $\mathcal{N}_{A_2}$ which is independent of $P_{A_2}$. Then,  $$\xi_{A_2}=\sum\limits_{k=1}^{P_{A_2}} \mathcal{N}_{A_2}^{(k)}.$$
With this construction,
$$ \xi_{A_2}(\cdot\cap A_1)=\displaystyle\sum\limits_{k=1}^{P_{A_2}} \textbf{1}\{\mathcal{N}^{(k)}_{A_2}(A_1)\geq 1 \}\mathcal{N}_{A_2}^{(k)}(\cdot\cap A_1).$$\\
By the previous Lemma \ref{lemi}, we only have to check that $\sum\limits_{k=1}^{P_{A_2}}\textbf{1}\{\mathcal{N}_{A_2}^{(k)}(A_1)\geq 1 \}$ has a Poisson law with parameter $X I_{A_1}$.
It is a Poissonian sum of independent Bernoulli random variables of parameter $\P(\mathcal{N}_{A_2}(A_1)\geq 1)$. It is well known that this is still a Poisson random variable with parameter $XI_{A_2}\times \P(\mathcal{N}_{A_2}(A_1)\geq 1)$. This parameter is equal to $XI_{A_1}$ by Lemma \ref{quot}.
\end{proof}
\section{Characterization of invariant measures} \label{sec5}
\paragraph{Strategy of the proof of theorem \ref{charac}}
Let us consider a cluster-invariant point process $\Theta$. Let $f\in\mathcal{F}^+_c(\R^d)$ whose support is included in a closed ball $A$.
By cluster-invariance, we remark that for every $n\in\N^*$,
$$\E\left[\exp\left({\displaystyle -\int f(x)\Theta(dx)}\right)\right]=\E\left[\exp\left({\displaystyle -\int f(x)\Theta_n(dx)}\right)\right]. $$
Then, we are looking for the asymptotics of the right-hand side above in order to describe the distribution of $\Theta$. Remark that particles $x$ of $\Theta$ such that $||x||\geq M\sqrt{n}$ can be neglected. That is why,
$$\E\left[\exp\left({\displaystyle -\int f(x)\Theta_n(dx)}\right)\right]\approx\E\left[\displaystyle\prod\limits_{x\in\Theta\cap B(0,M\sqrt{n})}\E_{\mathbf{P}}\left[\exp\left(-\sum\limits_{|u|=n}f(S_u+x)\right)\right]\right].$$
However, uniformly in $x\in B(0,M\sqrt{n})$,
\begin{align*}
\ln\left(\E_{\mathbf{P}}\left[\exp\left(-\sum\limits_{|u|=n}f(S_u+x)\right)\right]\right)\\
&\hspace{-5.5 cm}=\ln\left(1+\E_{\mathbf{P}}\left[\textbf{1}\{Z_n(A-x)\geq 1\}\left(\exp\left(-\sum\limits_{|u|=n}f(S_u+x)\right)-1\right)\right]\right)\\
&\hspace{-5.5 cm}\approx \E_{\mathbf{P}}\left[\textbf{1}\{Z_n(A-x)\geq 1\}\left(\exp\left(-\sum\limits_{|u|=n}f(S_u+x)\right)-1\right)\right]\\
&\hspace{-5.5 cm}\approx R_{f,A}\mathbf{P}(Z_n(A-x)\geq 1)
\end{align*}
where the last approximation stems from Lemma \ref{funlemma} and 
$$R_{f,A}=\E\left[\exp\left(-\int f(x)\mathcal{N}_A(dx)\right) \right]-1.$$
Finally, we get that for every $n\in\N^*$,
$$\E\left[\exp\left({\displaystyle -\int f(x)\Theta_n(dx)}\right)\right]\approx\E\left[\exp\left(R_{f,A}\int_{B(0,M\sqrt{n})}\mathbf{P}(Z_n(A-x)\geq 1)\Theta(dx)\right)\right].$$
In order to conclude the proof, we only need to find a subsequence along which $$L_n:=\int_{B(0,M\sqrt{n})}\mathbf{P}(Z_n(A-x)\geq 1)\Theta(dx)$$ converges in distribution. Therefore, it would be enough to show that $(L_n)_{n\in\N^*}$ is tight and use Prokhorov's theorem.
\subsection{Preliminary Lemmas}
In order to make the strategy above more rigorous, we need a few lemmas.
The proof of the following lemma is very similar to the proof of proposition 4.(b) in \cite{BCG}.
\begin{lem}\label{tightness}
Let $\Theta$ be a cluster-invariant point process. Let $A$ be a closed ball. Then, for every $n\in\N^*$,
$$\displaystyle\int\mathbf{P}(Z_n(A-x)\geq 1)\Theta(dx)$$
is finite almost surely. Moreover,
$$\left(\displaystyle\int\mathbf{P}(Z_n(A-x)\geq 1)\Theta(dx)\right)_{n\in\N}$$ is a tight sequence of random variables.
\end{lem}
\begin{proof}
Let $t>0$. By cluster-invariance of $\Theta$, it holds that
\begin{align}
\E\left[ \exp(-t\Theta(A))\right]&=\E\left[ \exp(-t\Theta_n(A))\right].\label{egafini1}
\end{align}
By definition of $\Theta_n$, (\ref{egafini1}) yields
\begin{align}
\E\left[ \exp(-t\Theta(A))\right]&=\E\left[\prod\limits_{x\in\Theta}\E_{\mathbf{P}}\left[\exp\left(-tZ_n(A-x)\right) \right] \right]\nonumber\\
&=\E\left[\exp\left(\int\ln\bigg(\E_{\mathbf{P}}\left[\exp\left(-tZ_n(A-x) \right) \right] \bigg) \Theta(dx) \right) \right].\label{egafini2}
\end{align}
However, using inequality $\ln(1+t)\leq t$ for every $t>-1$, one gets that for any $x\in\R^d$
\begin{align}
\ln\bigg(\E_{\mathbf{P}}\left[\exp\left(-tZ_n(A-x) \right) \right]\bigg)&
=\ln\bigg(1+\E_{\mathbf{P}}\left[\textbf{1}\{Z_n(A-x)\geq 1\}\bigg(\exp\left(-tZ_n(A-x) \right)-1\bigg) \right]\bigg)\nonumber\\
&\leq\E_{\mathbf{P}}\left[\textbf{1}\{Z_n(A-x)\geq 1\}\bigg(\exp\left(-tZ_n(A-x)\right)-1\bigg) \right]\nonumber\\
&\leq(e^{-t} -1)\mathbf{P}\left(Z_n(A-x)\geq 1 \right).\label{egafini3}
\end{align}
Combining (\ref{egafini2}) and (\ref{egafini3}), one gets that
\begin{align}
\E\left[ \exp(-t\Theta(A))\right]&\leq \E\left[ \exp\left((e^{-t}-1) \int\mathbf{P}\left(Z_n(A-x)\geq 1 \right) \Theta(dx)\right) \right]\nonumber\\
&=\E\left[ \exp\left((e^{-t}-1) W_n\right) \right]\label{egafini4}
\end{align}
where $$W_n:=\displaystyle\int \mathbf{P}(Z_n(A-x)\geq 1)\Theta(dx).$$
From (\ref{egafini4}), we deduce that for every $L>0$,
\begin{align}
\E\left[ \exp(-t\Theta(A))\right]&\leq 1-\P(W_n\geq L)+\P(W_n\geq L)\exp( (e^{-t}-1)L).\label{egafini5}
\end{align}
Therefore, for every $L>0$,
\begin{align}
\P\left(W_n\geq L\right)\leq \frac{1-\E\left[\exp(-t\Theta(A)) \right]}{1-\exp((e^{-t}-1)L)}.\label{egafini6}
\end{align}
By letting $L$ go toward infinity and $t$ go to 0 in (\ref{egafini6}), we get that
$$\P\left(W_n=+\infty\right)=0.$$
Moreover inequality (\ref{egafini6}) is clearly sufficient to prove tightness.
\end{proof}
\begin{lem} \label{recu}
Let $A$ be a closed ball. We assume hypotheses $\mathcal{H}_1$ or $\mathcal{H}_2$.
Then there exist non-negative constants $\kappa$ and $\lambda$ such that for every $M\geq1$, there exists an integer $N_M$ such that for every $n\geq N_M$ and for every $x\in B(0,M\sqrt{n})^c$,\\
$$\mathbf{P}(Z_n(A-x)\geq 1)\leq \kappa e^{-M^2/\kappa}\mathbf{P}(Z_{\lambda n}(A-x)\geq 1) .$$
\end{lem}
\begin{proof}
Let $(b_1,b_2,\cdots )$ be an infinite sequence of real numbers. Let $n\in\N^*$.
Obviously,
\begin{align}\label{egalem1}
\E_{\mathbf{Q}^*}\left[\left({1+\sum\limits_{k=1}^{n}{Y}_k(b_k)}\right)^{-1} \right]\leq 1
\end{align}
where ${Y}_k$ refers to the notation given in the proof of Lemma \ref{funlemma}.
Moreover, by Jensen's inequality,
$$\displaystyle\E_{\mathbf{Q}^*}\left[\left({1+\sum\limits_{k=1}^{n}{Y}_k(b_k)}\right)^{-1} \right]\geq \left({\E_{\mathbf{Q}^*}\left[1+\sum\limits_{k=1}^{n}{Y}_k(b_k)\right]}\right)^{-1}.$$\\
However, in a similar manner as in the proof of identity (\ref{ega1qua}),  it holds that
$$\E_{\mathbf{Q}^*}\left[\sum\limits_{k=1}^{n}{Y}_k(b_k)\right]\leq |A|c_d \sigma^2\sum\limits_{k=1}^{+\infty}\frac{1}{k^{d/2}}:=C_{A,d,\sigma}<+\infty.$$\\
Therefore, we get that for every sequence $(b_1,b_2,\cdots )$ and for every $n\in\N^*$,  \\
\begin{align}\label{egalem2}
\E_{\mathbf{Q}^*}\left[\left({1+\sum\limits_{k=1}^{n}{Y}_k(b_k)}\right)^{-1} \right]\geq \frac{1}{C_{A,d,\sigma}}.
\end{align}
However we can recall from the proof of Lemma \ref{funlemma} that for every $n\in\N^*$ and for any $x\in B(0,M\sqrt{n})^c$,\\
$$\begin{array}{ll}
\mathbf{P}(Z_n(A-x)\geq 1)&=\E_{\mathbf{Q}^*}\left[\frac{\textbf{1}\{S_{w_n}\in(A-x)\}}{1+\sum\limits_{k=1}^{n}Y_k(S_{w_{n-k}}+x)} \right]\\

\end{array}$$
where the spine $(S_{w_k})_{k\in\N^*}$ is a branching random walk with motion law $\mathcal{P}$. Moreover, by definition of $\mathbf{Q}^*$, $(S_{w_k})_{k\in\N^*}$ is independent of the random variables $(Y_k(\cdot))_{k\in\N^*}$ under $\mathbf{Q}^*$. Therefore, for every $n\in\N^*$ and for every $x\in B(0,M\sqrt{n})^c$,
\begin{align}
\mathbf{P}(Z_n(A-x)\geq 1)&=\E_{\mathbf{Q}^*}\left[\textbf{1}\{S_{w_n}\in(A-x)\}\E_{\mathbf{Q}^*}\left[\left({1+\sum\limits_{k=1}^{n}Y_k(b_k)}\right)^{-1} \right]_{b_k=S_{w_{n-k}}+x}\right]. \label{condi}
\end{align}
Combining identities (\ref{egalem1}), (\ref{egalem2}) and (\ref{condi}) yields
\begin{align}
C_{A,d,\sigma}^{-1}\mathbf{Q}^*\left(S_{w_n}\in A-x \right)\leq \mathbf{P}(Z_n(A-x)\geq 1)\leq \mathbf{Q}^*\left(S_{w_n}\in A-x \right).\label{branchtonorm}
\end{align}
Consequently, under the hypothesis $\mathcal{H}_1$, for every $n\in\N^*$ and for every $x\in B(0,M\sqrt{n})^c$,\\
$$\displaystyle\frac{\mathbf{P}\left( Z_n(A-x)\geq 1\right)}{\mathbf{P}\left( Z_{2n}(A-x)\geq 1\right)}\leq \frac{\displaystyle C_{A,d,\sigma}\int_Ae^{-||y-x||^2/(2n)}dy}{\displaystyle \int_Ae^{-||y-x||^2/(4n)}dy}.$$\\
Therefore, under the hypothesis $\mathcal{H}_1$, there exists a positive constant $C'_{A,d,\sigma}$ such that for every $n\in\N^*$ and for every $x\in B(0,M\sqrt{n})^c$,
\begin{align}
\displaystyle\frac{\mathbf{P}\left( Z_n(A-x)\geq 1\right)}{\mathbf{P}\left( Z_{2n}(A-x)\geq 1\right)}&\leq C'_{A,d,\sigma}e^{-||x||^2/(4n)}\nonumber\\
&\leq C'_{A,d,\sigma}e^{-M^2/4}.\label{casH1}
\end{align}
Thus, (\ref{casH1}) proves Lemma \ref{recu} under the hypothesis $\mathcal{H}_1$. Under the hypothesis $\mathcal{H}_2$, we will use Proposition \ref{heat} to give a similar proof. Let us use constants $\tau$, $C_1$ and $C_2$ introduced in Proposition \ref{heat}. Moreover, under hypothesis $\mathcal{H}_2$, $\mathcal{P}$ has a finite range. Therefore there exists a constant $C_3$ such that $\mathbf{P}\left( Z_n(A-x)\geq 1\right)=0$ for every $n\in\N^*$ and for every $x\in B(0, C_3n)^c$. Now let us define $r=\max(C_2 C_3/(\tau C_1),2)$. Let $n$ be an integer larger than $(M/C_3)^{2}$ and let $x$ be such that $M\sqrt{n}< ||x||\leq C_3n$. Then, we get
\begin{align}
\mathbf{P}(Z_n(A-x)\geq 1)&\leq\sum\limits_{y\in A\cap\mathbb{Z}^d} \mathbf{Q}^*\left( S_{w_n}=y-x\right)\nonumber\\
&\leq C_1 n^{-d/2}\sum\limits_{y\in A\cap\mathbb{Z}^d}\exp\left( -\frac{||x-y||^2}{C_1n}\right)
\end{align}
where we used identity (\ref{branchtonorm}) in the first inequality and Proposition \ref{heat} in the second one.
Then, as $||x||\geq M\sqrt{n}$, we get
\begin{align}
\mathbf{P}(Z_n(A-x)\geq 1)&\leq  C\times n^{-d/2}\exp\left(-\frac{(r-1)}{C_1r}M^2 \right)\sum\limits_{y\in A\cap\mathbb{Z}^d}\exp\left(-\frac{||x-y||^2}{C_2( rC_1C_2^{-1}n) }\right)
\end{align}
where $C$ depends only on $A$ and $C_1$.
However, by definition of $r$, $||x||\leq C_3 n\leq \tau r C_1 C_2^{-1} n$. Therefore, by the lower bound in Proposition \ref{heat}, there exists $C'>0$ such that for every $x$ such that $M\sqrt{n}< ||x||\leq C_3n$,
\begin{align}
\mathbf{P}(Z_n(A-x)\geq 1)&\leq C'\times \exp\left(-\frac{(r-1)}{C_1r}M^2 \right)\sum\limits_{y\in A\cap\mathbb{Z}^d} \mathbf{Q}^*\left( S_{w_{r C_1 C_2^{-1} n}}=y-x\right)\nonumber\\
&\leq C'\times C_{A,d,\sigma}\times \exp\left(-\frac{(r-1)}{C_1r}M^2 \right) \mathbf{P}\left(Z_{r C_1 C_2^{-1} n}(A-x)\geq 1 \right) \label{fin}
\end{align}
where  we used identity (\ref{branchtonorm}) in the second inequality. If $||x||\geq C_3n$, then the left-hand side in (\ref{fin}) is zero. Thus (\ref{fin}) is also true in this case. This concludes the proof.
\end{proof}
\begin{lem}
\label{majoc1}
Let $A$ be a closed ball. Let us assume hypotheses $\mathcal{H}_1$ or $\mathcal{H}_2$. Let $\Theta$ be a cluster-invariant point process. Then,
$$
\displaystyle\underset{M\rightarrow+\infty}\lim \hspace{0.2 cm} \underset{n\rightarrow+\infty}\limsup \hspace{0.2 cm}\P\left(\int_{B(0,M\sqrt{n})^c}\mathbf{P}(Z_n(A-x)\geq 1)\Theta(dx)\geq e^{-M^2/(2\kappa)} \right)=0.$$

\end{lem}
\begin{proof}
By Lemma \ref{recu}, for every $M\geq 1$ and for every $n\geq N_M$,\\
$$\begin{array}{l}
\displaystyle\P\left(\int_{ B(0,M\sqrt{n})^c}\mathbf{P}(Z_n(A-x)\geq 1))\Theta(dx)\geq e^{-M^2/(2\kappa)} \right)\\ \hspace{3 cm}\leq \displaystyle\P\left(\int_{ B(0,M\sqrt{n})^c}\mathbf{P}(Z_{\lambda n}(A-x)\geq 1)\Theta(dx)\geq \kappa^{-1}e^{M^2/(2\kappa)} \right)\\
\hspace{3 cm}\leq \displaystyle\P\left(\int\mathbf{P}(Z_{\lambda n}(A-x)\geq 1)\Theta(dx)\geq \kappa^{-1}e^{M^2/(2\kappa)} \right).
\end{array}$$
However,  by Lemma \ref{tightness},
$$\left(\int\mathbf{P}(Z_{\lambda n}(A-x)\geq 1)\Theta(dx)\right)_{n\in\N} $$
is tight. This concludes the proof.
\end{proof}

\subsection{Proof of Theorem \ref{charac}}
\begin{proof}[Proof of Theorem \ref{charac}]
Let $A$ be a closed ball.
Let $f\in\mathcal{F}_c^+(\R^d)$ such that $supp(f)\subset A$. Let $M>0$.
By cluster-invariance of $\Theta$, for any $n\in\N^*$,\\
\begin{align} \label{egah0}
\E\left[\exp\left({\displaystyle -\int f(x)\Theta(dx)}\right)\right]&=\E\left[\exp\left({\displaystyle -\int f(x)\Theta_n(dx)}\right)\right]\nonumber\\
&=\E\left[\displaystyle\prod\limits_{x\in\Theta}\E_{\mathbf{P}}\left[\exp\left(-\sum\limits_{|u|=n}f(S_u+x)\right)\right] \right] \nonumber\\
&=\E\left[\displaystyle\prod\limits_{x\in\Theta\cap B(0,M\sqrt{n})}\E_{\mathbf{P}}\left[\exp\left(-\sum\limits_{|u|=n}f(S_u+x)\right)\right]\right.\nonumber\\
&\hspace{0.4 cm}\times\left.\displaystyle\prod\limits_{x\in\Theta\cap B(0,M\sqrt{n})^c}\E_{\mathbf{P}}\left[\exp\left(-\sum\limits_{|u|=n}f(S_u+x)\right)\right]\right].
\end{align}\\
First, let us look at the random variable
\begin{align}J_{n,M}:=\displaystyle\int_{ B(0,M\sqrt{n})^c}\ln\left(\E_{\mathbf{P}}\left[\exp\left(-\sum\limits_{|u|=n}f(S_u+x)\right)\right] \right) \Theta(dx). \label{quant1}
\end{align}
As $supp(f)\subset A$,
\begin{align}
J_{n,M}=\displaystyle\int_{ B(0,M\sqrt{n})^c}\ln\left(1+\E_{\mathbf{P}}\left[\textbf{1}\{Z_n(A-x)\geq 1\}\left(\exp\left(-\sum\limits_{|u|=n}f(S_u+x)\right)-1\right)\right] \right)\Theta(dx).\label{quant2}
\end{align}
Moreover, we remark that for every $x\in\R^d$,
$$\begin{array}{ll}\Bigg|\E_{\mathbf{P}}\left[\textbf{1}\{Z_n(A-x)\geq 1\}\left(\exp\left(-\sum\limits_{|u|=n}f(S_u+x)\right)-1\right)\right]\Bigg|&\leq \mathbf{P}\left(Z_n(A-x)\geq 1 \right) \\ 
&\leq \E_{\mathbf{P}}\left[Z_n(A-x) \right] \vspace{0.2 cm}\\ 
&=\P(\hat{S}_n\in A-x)
\end{array}$$
where $(\hat{S}_n)_{n\in\N}$ is a random walk with motion law $\mathcal{P}$. Together with Lemma \ref{inegastand}, this implies:
\begin{align}
\underset{x\in\R^d}\sup\Bigg|\E_{\mathbf{P}}\left[\textbf{1}\{Z_n(A-x)\geq 1\}\left(\exp\left(-\sum\limits_{|u|=n}f(S_u+x)\right)-1\right)\right]\Bigg|\leq c_d\frac{|A|}{n^{d/2}}\xrightarrow[n\rightarrow +\infty]{} 0.\label{borneunif}
\end{align}
Therefore, as $|\ln(1+t)|\leq 2|t|$ for every $t$ small enough, the combination of identities (\ref{quant2}) and (\ref{borneunif}) implies that there exists an integer $\tilde{N}$ such that for every $n\geq \tilde{N}$,\\
$$|J_{n,M}|\leq 2\int_{ B(0,M\sqrt{n})^c}\mathbf{P}(Z_n(A-x)\geq 1)\Theta(dx)$$
almost surely. Therefore, using Lemma \ref{majoc1},
\begin{align}
\P\left(|J_{n,M}|\geq 2e^{-M^2/(2\kappa)}\right)
&= o_{n,M}(1)\label{egah1}
\end{align}
where $o_{n,M}(1)$ is a function $\left(\varepsilon_{n,M} \right)_{n\in\N^*,M\in\R_+^*}$ such that $\underset{M\rightarrow +\infty}\lim \underset{n\rightarrow +\infty}\limsup \hspace{0.2 cm}\varepsilon_{n,M}=0$.
Furthermore, by (\ref{egah0}), for every $n\geq \tilde{N}$,
\begin{align}
\E\left[\exp\left({\displaystyle -\int f(x)\Theta(dx)}\right)\right] \nonumber\\
&\hspace{-4 cm}=\E\left[\displaystyle\prod\limits_{x\in\Theta\cap B(0,M\sqrt{n})}\E_{\mathbf{P}}\left[\exp\left(-\sum\limits_{|u|=n}f(S_u+x)\right)\right]\times \exp(J_{n,M})\right]\nonumber\\
&\hspace{-4 cm}=\E\left[\displaystyle\prod\limits_{x\in\Theta\cap B(0,M\sqrt{n})}\E_{\mathbf{P}}\left[\exp\left(-\sum\limits_{|u|=n}f(S_u+x)\right)\right]\right]\nonumber\\
&\hspace{-3.6 cm}+\E\left[\displaystyle\prod\limits_{x\in\Theta\cap B(0,M\sqrt{n})}\E_{\mathbf{P}}\left[\exp\left(-\sum\limits_{|u|=n}f(S_u+x)\right)\right]\times(\exp(J_{n,M})-1)\right].\label{egah2}
\end{align}
However, 
\begin{align}
\Bigg|\E\left[\displaystyle\prod\limits_{x\in\Theta\cap B(0,M\sqrt{n})}\E_{\mathbf{P}}\left[\exp\left(-\sum\limits_{|u|=n}f(S_u+x)\right)\right]\times(\exp(J_{n,M})-1)\right]\Bigg|\nonumber
\end{align}
is lower than,
\begin{align}
\E\left[|\exp(J_{n,M})-1| \right]&=\E\left[|\exp(J_{n,M})-1|\textbf{1}\{|J_{n,M}|<2e^{-M^2/(2\kappa)}\} \right]\nonumber\\
&\hspace{0.4 cm}+\E\left[|\exp(J_{n,M})-1|\textbf{1}\{|J_{n,M}|\geq2e^{-M^2/(2\kappa)}\} \right]\nonumber\\
&\leq o_{n,M}(1)+\P(|J_{n,M}|\geq2e^{-M^2/(2\kappa)})\nonumber\\
&=o_{n,M}(1)\label{majoJ}
\end{align}
where the second inequality comes from the fact that $J_{n,M}$ is non-positive and the last equality comes from identity (\ref{egah1}). Consequently, combining (\ref{egah2}) and (\ref{majoJ}), it holds that, 
\begin{align}
\E\left[\exp\left({\displaystyle -\int f(x)\Theta(dx)}\right)\right]&=\E\left[\displaystyle\prod\limits_{x\in\Theta\cap B(0,M\sqrt{n})}\E_{\mathbf{P}}\left[\exp\left(-\sum\limits_{|u|=n}f(S_u+x)\right)\right]\right]+o_{n,M}(1).\label{egah3}
\end{align}
Furthermore, as $supp(f)\subset A$,
\begin{align}\label{egah4}
\hspace{-0.3 cm}\E\left[\displaystyle\prod\limits_{x\in\Theta\cap B(0,M\sqrt{n})}\E_{\mathbf{P}}\left[\exp\left(-\sum\limits_{|u|=n}f(S_u+x)\right)\right]\right]=\E\left[\exp\left(\int_{ B(0,M\sqrt{n})} h_n(x)\Theta(dx)\right)\right]
\end{align}
where $h_n(x)=\ln\left(1+\E_{\mathbf{P}}\left[\textbf{1}\{ Z_n(A-x)\geq 1\}\left(\exp\left({\displaystyle -\sum\limits_{|u|=n}f(S_u+x)}\right)-1 \right) \right] \right)$.\\
Nevertheless, by Lemma \ref{funlemma}, uniformly in $x\in B(0,M\sqrt{n})$,
\begin{align}
h_n(x)=(1+o_n(1))R_{f,A}\mathbf{P}(Z_n(A-x)\geq 1)\label{egah5}
\end{align}
where we recall that $$\displaystyle R_{f,A}=\E\left[\exp\left(-\int f(x)\mathcal{N}_A(dx) \right)\right]-1=\frac{I_{A,-f}}{I_A}-1.$$ Consequently, using (\ref{egah5}) in (\ref{egah4}) yields
\begin{align}
\E\left[\displaystyle\prod\limits_{x\in\Theta\cap B(0,M\sqrt{n})}\E_{\mathbf{P}}\left[\exp\left(-\sum\limits_{|u|=n}f(S_u+x)\right)\right]\right]\nonumber\\
&\hspace{-7 cm}=\E\left[\exp\left((1+o_n(1))R_{f,A}\int_{ B(0,M\sqrt{n})}\mathbf{P}(Z_n(A-x)\geq 1) \Theta(dx)\right)\right]. \label{egah6}
\end{align}
Moreover, by Lemma \ref{tightness}, we know that 
$$\int_{ B(0,M\sqrt{n})}\mathbf{P}(Z_n(A-x)\geq 1)\Theta(dx)$$ 
is tight.
Therefore, by Prokhorov's Theorem (see for example \cite{Billingsley_cv}), there exists a subsequence $r_n=r_n(M,A)$ such that $$\int_{ B(0,M\sqrt{r_n})}\mathbf{P}(Z_{r_n}(A-x)\geq 1)\Theta(dx)$$ converges in law toward some random variable $Y_{M,A}$.
Combining this with (\ref{egah3}), (\ref{egah6}) and letting $n$ go toward $+\infty$, we deduce that\\
\begin{align}
\E\left[\exp\left({\displaystyle -\int f(x)\Theta(dx)}\right)\right] &=\E\left[\exp\left(Y_{M,A}R_{f,A}\right)\right]+o_M(1). \label{egah7}
\end{align}
Besides, thanks to Lemma \ref{tightness}, the tightness of $(Y_{M,A})_{M\in\R_+^*}$ is easily checked. Thus a subsequence of $(Y_{M,A})_{M\in\N^*}$ converges in law toward some finite, non-negative random variable $Y_A$. Letting $M$ go to infinity in (\ref{egah7}), we deduce that for every $f\in\mathcal{F}_c^+(\R^d)$ such that $supp(f)\subset A$,\\
\begin{align}\label{egah9}\E\left[\exp\left({- \displaystyle\int f(x)\Theta(dx)}\right)\right]&=\E\left[\exp( Y_AR_{f,A}) \right]\nonumber\\
&=\E\left[\exp( X_AI_AR_{f,A}) \right]
\end{align}
with $X_A:=Y_A\times I_A^{-1}$ where $I_A$ is defined in definition \ref{defiprat}. Then, we can compare the expression of the Laplace transform given in Proposition \ref{laplace} and the expression given by identity (\ref{egah9}) and we get that for every $f\in\mathcal{F}_c^+(\R^d)$ such that $supp(f)\subset A$:
\begin{align}
\E\left[\exp\left({- \displaystyle\int f(x)\Theta(dx)}\right)\right]=\E\left[\exp\left({-\displaystyle\int f(x)\Lambda^{d,X_A}(dx)}\right)\right].\nonumber
\end{align}
Consequently, by Lemma \ref{egaloi}, $\Theta(\cdot\cap A)$ has the same law as $\Lambda^{d,X_A}(\cdot\cap A)$. Thus, we almost proved Theorem \ref{charac}. The only problem is that we only proved our result in $A$. Of course we could consider larger and larger balls $A$ but we have to verify that $A\subset B$ implies $X_A\overset{law}=X_B$. 
Thus, let us consider two closed balls $A$ and $B$ such that $A\subset B$. Let $f\in\mathcal{F}_c^+(\R^d)$ such that $supp(f)\subset A$. Then we can equalize the two formulas for the Laplace transform given by (\ref{egah9}):
\begin{align}
\E\left[\exp( X_A I_A R_{f,A}) \right]=\E\left[\exp\left(- \displaystyle\int f(x)\Theta(dx)\right)\right]=\E\left[\exp( X_B I_BR_{f,B}) \right].\label{egah95}
\end{align}
As $supp(f)\subset A$, we know that
\begin{align}
R_{f,B}&=\E\left[\exp\left(-\int f(x)\mathcal{N}_B(dx)\right)-1 \right]\nonumber\\
&=\E\left[\textbf{1}\{\mathcal{N}_B(A)\geq 1 \}\left( \exp\left(-\int f(x)\mathcal{N}_B(dx)\right)-1\right)\right].\label{egah10}
\end{align}
Therefore, using (\ref{egah10}) and Lemmas \ref{quot} and \ref{lemi}, we obtain that
\begin{align}
\frac{I_B}{I_A}R_{f,B}&=\E\left[\exp\left(-\int f(x)\mathcal{N}_B(dx)\right)-1\Bigg| \mathcal{N}_B(A)\geq 1 \right]\nonumber\\
&=\E\left[\exp\left(-\int f(x)\mathcal{N}_A(dx)\right)\right]-1\nonumber\\
&= R_{f,A}.
\end{align}
Thus $I_AR_{f,A}=I_BR_{f,B}$. Together with (\ref{egah95}), for every $f\in\mathcal{F}_c^+(\R^d)$ such that $supp(f)\subset A$, this implies that
$$\E\left[\exp( X_A I_A R_{f,A}) \right]=\E\left[\exp( X_B I_AR_{f,A}) \right].$$
Moreover $$I_AR_{f,A}=I_A\times \left( \E\left[\exp\left(-\int f(x)\mathcal{N}_A(dx)\right)\right]-1 \right).$$
It is equal to $0$ for $f=0$ and $-I_AR_{f,A}$ can be made as close as we want from $I_A$ by choosing $f$ of larger and larger infinite norm. Consequently, for every $t\in [0, I_A[$, $$\E\left[ \exp(-tX_A)\right]=\E\left[ \exp(-tX_B)\right].$$
It is sufficient to say that $X_A$ and $X_B$ have the same distribution because $I_A>0$. It concludes the proof.
\end{proof}

\section{Heavy tail genereralization } \label{gene}
\subsection{New setting}
In this section, we prove Theorem \ref{general} which states that Theorems \ref{cv} and \ref{charac} remain true under the hypothesis $\mathcal{H}_3$. We recall that hypothesis $\mathcal{H}_3$ consists in the following assumptions:
The motion law $\mathcal{P}$ is given by spherically symmetric $\alpha$-stable laws with $\alpha\in]0,2[$. More precisely, for every $y\in\R^d$, $$\int\exp\left(i\langle y,x \rangle \right)\mathcal{P}(dx)=\exp\left(-\left(\sum\limits_{k=1}^d|y_k|^2\right)^{\alpha/2}\right).$$ The critical offspring law $\mu$ does not have second moment anymore. However we assume that there exists $\beta\in]0,1]$ such that for every $\gamma<\beta$, $\sum_{k=0}^{+\infty} k^{1+\gamma}\mu(k)<+\infty$. Moreover we assume that $d>\alpha/\beta$.

In this new setting, $\alpha\in ]0,2[$. Of course, $\alpha=2$ is the Brownian case which has been treated under the hypothesis $\mathcal{H}_1$. Assuming hypothesis $\mathcal{H}_3$, any particle is following a spherically symmetric $\alpha$-stable Lévy process until its death. One can refer to \cite{Sato} for more informations about general Lévy processes.

\begin{rem}
The typical example of an offspring law $\mu$ satisfying hypothesis $\mathcal{H}_3$ is given by $f(s)=s+\frac{1}{2}(1-s)^{1+\beta}$ with $\beta\in]0,1[$, where $f(s)=\sum\limits_{k=0}^{+\infty} \mu(k) s^k$ for every $s\in[0,1]$.
\end{rem}

Of course, it would be much too long to modify the entirety of our proofs. In this section, we just indicate to the reader how to change the key steps in the proofs of Theorem \ref{cv} and \ref{charac}. In the sequel, we often use a random variable $\chi$ such that for every $y\in\R^d$, $$\E\left[e^{i \langle y, \chi\rangle} \right]=\exp\left(-\left(\sum\limits_{k=1}^d|y_k|^2\right)^{\alpha/2}\right).$$
Let us call such a random variable $\chi$ "a standard  ($d$-dimensional) $\alpha$-stable law". We denote by $L_{\alpha}$ the density of $\chi$. Except for a few values of $\alpha$, $L_{\alpha}$ has not any closed form.\\

\subsection{Generalized proofs}
In order to prove generalized versions of Theorem \ref{cv} we have to prove a new version of Lemma \ref{funlemma}. Under the hypothesis $\mathcal{H}_3$, $I_A$ and $I_{A,f}$ are defined as in definition \ref{defiprat}.
\begin{lem}[ Key Lemma - generalized version] \label{modififunlemma}
We assume hypothesis $\mathcal{H}_3$.
Let $A$ be a closed ball. Let $f\in\mathcal{F}_c(\R^d)$ whose values are in $i\R$ or in $\R_-$ and such that $supp(f)\subset A$. 
Let $M>0$.
Uniformly in $x\in B(0,Mn^{1/\alpha})$, as $n$ goes to infinity,
$$\displaystyle\E_{\mathbf{P}}\left[\textbf{1}\{Z_n(A-x)\geq 1\}\exp\left(\sum\limits_{|u|=n} f(S_u+x) \right) \right]=(1+o_n(1)) \frac{L_{\alpha}(x/n^{1/\alpha})}{n^{d/\alpha}}I_{A,f} .$$
\end{lem}
\begin{rem}
$I_{A,f}$ is defined as in definition \ref{defiprat} but the motion is now of Levy type. The well-definiteness of $I_{A,f}$ can be proved exactly as in the following proof of Lemma \ref{modififunlemma}.
\end{rem}
Now, let us prove Lemma \ref{modififunlemma}.
\begin{proof}
With the same notation as in the proof of Lemma \ref{funlemma} we know that
$$\begin{array}{ll}
\E_{\mathbf{P}}\left[\textbf{1}\{Z_n(A-x)\geq 1\}\exp\left(\sum\limits_{|u|=n} f(S_u+x) \right) \right]&\nonumber \vspace{0.2 cm} \\
&\hspace{-7 cm}=\E_{\mathbf{Q}^*}\left[\frac{\textbf{1}\{S_{w_n}\in A-x\}}{1+\sum\limits_{k=1}^{n}Y_{k}(S_{w_{n-k}} +x)}\exp\left(\sum\limits_{k=0}^{n}Y_{f,k}(S_{w_{n-k}}+x)\right) \right] .
\end{array}$$
Now let us prove that we can resritct the sum above to a finite number of terms, as in the proof of Lemma \ref{funlemma}. Let $K>0$. Let us consider $\gamma>0$ such that $\gamma<\beta\leq 1$. Recall the notation $\E_{\mathbf{Q}^*}\left[\cdot| \mathcal{G} \right]$ when we condition on the spine and on the number of brothers of every particle in the spine and the positions of the brothers of the spine.
Then, by conditional Markov inequality,\\
$$\begin{array}{ll}
\displaystyle\mathbf{Q}^*\left(S_{w_n}\in A-x,\sum\limits_{k=K+1}^{n}Y_k(x+S_{w_{n-k}})\geq 1\right)\\
&\hspace{-7 cm}\leq\displaystyle\E_{\mathbf{Q}^*}\left[\textbf{1}\{S_{w_n}\in A-x\}\E_{\mathbf{Q}^*}\left[\left(\sum\limits_{k=K+1}^{n}Y_k(x+S_{w_{n-k}})\right)^{\gamma} \Bigg|\mathcal{G}\right] \right].\\

\end{array}$$
We now that for any non negative numbers $(\lambda_i)_{1\leq i \leq n}$ , $\left(\sum\limits_{i=1}^n\lambda_i\right)^{\gamma}\leq \sum\limits_{i=1}^n\lambda_i^{\gamma}$ because $0<\gamma<1$.
Consequently,
\begin{align}
\displaystyle\mathbf{Q}^*\left(S_{w_n}\in A-x,\sum\limits_{k=K+1}^{n}Y_k(x+S_{w_{n-k}})\geq 1\right)\nonumber\\
&\hspace{-7 cm}\leq \displaystyle\E_{\mathbf{Q}^*}\left[\textbf{1}\{S_{w_n}\in A-x\}\sum\limits_{k=K+1}^{n}\E_{\mathbf{Q}^*}\left[Y_k(x+S_{w_{n-k}})^{\gamma} \Bigg| \mathcal{G}\right] \right].\label{inegag1}
\end{align}
However, for every $k\in\{K+1,\cdots, n\}$, by definition of $Y_k(\cdot)$,
\begin{align}
\E_{\mathbf{Q}^*}\left[Y_k(x+S_{w_{n-k}})^{\gamma} \Bigg| \mathcal{G}\right] &=\E_{\mathbf{Q}^*}\left[\left(\sum\limits_{u\in B(w_{n-k+1})} Z^u_{k-1}(A-x-S_{w_{n-k}}-\rho_u) \right)^{\gamma} \Bigg| \mathcal{G}\right]\nonumber
\end{align}
where for every $u\in B(w_{n-k+1})$, recall that $\rho_u=S_{u}-S_{w_n}$.
Then, as $0<\gamma<1$, by Jensen's inequality,
\begin{align}
\E_{\mathbf{Q}^*}\left[Y_k(x+S_{w_{n-k}})^{\gamma} \Bigg| \mathcal{G}\right] &\leq\E_{\mathbf{Q}^*}\left[\sum\limits_{u\in B(w_{n-k+1})} Z^u_{k-1}(A-x-S_{w_{n-k}}-\rho_u) \Bigg| \mathcal{G}\right]^{\gamma}\nonumber\\
&=\left(\sum\limits_{u\in B(w_{n-k+1})}\E_{\mathbf{Q}^*}\left[Z^u_{k-1}(A-x-S_{w_{n-k}}-\rho_u)\Bigg| \mathcal{G} \right]\right)^{\gamma}.\label{inegag2}
\end{align}
However, by construction of $\mathbf{Q}^*$, for every $k\in\{K+1,\cdots ,n\}$ and for every $u\in B(w_{n-k+1})$, it holds that
$$\E_{\mathbf{Q}^*}\left[Z^u_{k-1}(A-x-S_{w_{n-k}}-\rho_u)\Bigg| \mathcal{G} \right]=\E_{\mathbf{P}}\left[Z_{k-1}(A-z-\rho) \right]|_{z=x+S_{w_{n-k}},\rho=\rho_u}. $$
Then, by criticality of the branching process under $\mathbf{P}$, for every $k\in\{K+1,\cdots ,n\}$ and for every $u\in B(w_{n-k+1})$,
\begin{align}
\E_{\mathbf{Q}^*}\left[Z^u_{k-1}(A-x-S_{w_{n-k}}-\rho_u)\Bigg| \mathcal{G} \right]&=\P\bigg((k-1)^{1/\alpha}\chi\in A-z-\rho\bigg)|_{z=x+S_{w_{n-k}},\rho=\rho_u}\nonumber
\end{align}
where $\chi$ is a standard $d$-dimensional $\alpha$-stable law.
Then, as the density $L_{\alpha}$ is bounded, there exists a constant $c_{\alpha}>0$ such that for every $k\in\{K+1,\cdots ,n\}$ and for every $u\in B(w_{n-k+1})$,
\begin{align}
\E_{\mathbf{Q}^*}\left[Z^u_{k-1}(A-x-S_{w_{n-k}}-\rho_u)\Bigg| \mathcal{G} \right]&=\int_{A/(k-1)^{1/\alpha}}L_{\alpha}\left(y-\frac{x+S_{w_{n-k}}+\rho_u}{(k-1)^{1/\alpha}}\right) dy\nonumber\\
&\leq c_{\alpha}|A|k^{-d/\alpha}.\label{inegag3}
\end{align}
 Combining, (\ref{inegag1}), (\ref{inegag2}) and (\ref{inegag3}), we get

\begin{align}
\displaystyle\mathbf{Q}^*\left(S_{w_n}\in A-x,\sum\limits_{k=K+1}^{n}Y_k(x+S_{w_{n-k}})\geq 1\right)\nonumber\\
&\hspace{-7 cm}\leq c_{\alpha}^{\gamma} |A|^{\gamma}\sum\limits_{k={K+1}}^nk^{-d\gamma/\alpha}\E_{\mathbf{Q}^*}\left[\textbf{1}\{S_{w_n}\in A-x \} |B(w_{n-k+1})|^{\gamma}\right].\label{inegag4}
\end{align}
However for every $k\in\{1,\cdots, n\}$, $1+|B(w_{n-k+1})|$, that is, the number of children of $w_{n-k}$, has the distribution $\nu$ and is independent of the spine under $\mathbf{Q}^*$. We recall that $\nu$ is defined by $\nu(k)=k\mu(k)$. We assumed that $\int x^{1+\gamma}\mu(dx)<+\infty$. Thus, $\hat{c}_{\gamma}:=\int x^{\gamma}\nu(dx)<+\infty$. Therefore, together with (\ref{inegag4}), this yields

\begin{align}
\displaystyle\mathbf{Q}^*\left(S_{w_n}\in A-x,\sum\limits_{k=K+1}^{n}Y_k(x+S_{w_{n-k}})\geq 1\right)&\leq c_{\alpha}^{\gamma} |A|^{\gamma}\hat{c}_{\gamma}\sum\limits_{k={K+1}}^nk^{-d\gamma/\alpha}\times \mathbf{Q}^*\left( S_{w_n}\in A-x\right)\nonumber\\
&\leq c_{\alpha}^{1+\gamma}|A|^{1+\gamma}\hat{c}_{\gamma}n^{-d/\alpha}\sum\limits_{k={K+1}}^{+\infty}k^{-d\gamma/\alpha}.
\end{align}

Now, we need the sum to be convergent. For this, we need $d\gamma/\alpha>1$, that is, $d>\frac{\alpha}{\gamma}$. However we assumed that $d>\frac{\alpha}{\beta}$. Thus, this is possible to find such a $\gamma<\beta$. Therefore,  uniformly in $x\in B(0,Mn^{1/\alpha})$, it holds that
\begin{align}
\E_{\mathbf{P}}\left[\textbf{1}\{Z_n(A-x)\geq 1\}\exp\left(\sum\limits_{|u|=n} f(S_u+x) \right) \right]&\nonumber \vspace{0.2 cm} \\
&\hspace{-7 cm}=\E_{\mathbf{Q}^*}\left[\frac{\textbf{1}\{S_{w_n}\in A-x\}}{1+\sum\limits_{k=1}^{K}Y_{k}(S_{w_{n-k}} +x)}\exp\left(\sum\limits_{k=0}^{K}Y_{f,k}(S_{w_{n-k}}+x)\right) \right] +o_K(1)n^{-d/\alpha}.\label{rustine1}
\end{align}
Recall that for every $m\in\N^*$ the density of $S_{w_m}$ is given by $z\mapsto m^{-d/{\alpha}}L_{\alpha}(m^{-1/\alpha}z)$. Therefore, using (\ref{rustine1}) and following the same lines as in Lemma \ref{funlemma}, we get
\begin{align}
\E_{\mathbf{P}}\left[\textbf{1}\{Z_n(A-x)\geq 1\}\exp\left(\sum\limits_{|u|=n} f(S_u+x) \right) \right]&\nonumber \vspace{0.2 cm} \\
&\hspace{-7 cm}=(n-K)^{-d/\alpha}\E_{\mathbf{Q}^*}\left[\int_AL_{\alpha}((n-K)^{-1/\alpha}(y-x-S_{w_K}))F_{K,A}(y,S_{w_k},1\leq k\leq K)dy \right]\nonumber\\
&\hspace{-6.6 cm} +o_K(1)n^{-d/\alpha}.\label{rustine2}
\end{align}
Remark that
\begin{align}
\hspace{-0.4 cm}(n-K)^{-1/\alpha}(y-x-S_{w_K})&=-n^{-1/\alpha}x+ \underbrace{\frac{1-(1-K/n)^{1/\alpha}}{(n-K)^{1/\alpha}}x+ (n-K)^{-1/\alpha}(y-S_{W_K})}_{\varepsilon_{n,K,x,y}}.\label{rustine3}
\end{align}
On the event $\{||S_{W_k}||\leq K^{2/\alpha}\}$, uniformly in $x\in B(0,Mn^{1/\alpha})$, the error term $\varepsilon_{n,K,x,y}$ can be majorized in the following way

\begin{align}
||\varepsilon_{n,K,x,y}||&\leq M\times o_{n,K}(1) +(\sup_{z\in A}||z||)\times(n-K)^{-1/\alpha} +K^{2/\alpha}(n-K)^{-1/\alpha}.\nonumber\\
&=o_{n,K}(1).\label{rustine3}
\end{align}
Moreover, $L_{\alpha}$ is continuous and goes to 0 at infinity. Therefore, this function is uniformly continuous. This implies that, uniformly in $x\in B(0,Mn^{1/\alpha})$, on the event $\{||S_{W_k}||\leq K^{2/\alpha}\}$,
\begin{align}
|L_{\alpha}(n^{-1/\alpha}x)-L_{\alpha}((n-K)^{-1/\alpha}(y-x-S_{w_K}))|=o_{n,K}(1)\label{rustine4}
\end{align}
where $L_{\alpha}(-n^{-1/\alpha}x)$ was replaced by $L_{\alpha}(n^{-1/\alpha}x)$ by symmetry of $L_{\alpha}$.
Furthermore, the event $\{||S_{w_K}||\geq K^{2/\alpha}\}$ can be neglected because
\begin{align}
\mathbf{Q}^*(||S_{w_K}||\geq K^{2/\alpha})=o_K(1).\label{rustine5}
\end{align}
Therefore, combining (\ref{rustine2}), (\ref{rustine4}) and (\ref{rustine5}), we get that, uniformly in $x\in B(0,Mn^{1/\alpha})$,
\begin{align}
\E_{\mathbf{P}}\left[\textbf{1}\{Z_n(A-x)\geq 1\}\exp\left(\sum\limits_{|u|=n} f(S_u+x) \right) \right]&\nonumber \vspace{0.2 cm} \\
&\hspace{-7 cm}=n^{-d/\alpha}L_{\alpha}(n^{-1/\alpha}x)\E_{\mathbf{Q}^*}\left[\int_AF_{K,A}(y,S_{w_k},1\leq k\leq K) dy\right]\nonumber\\
&\hspace{-6.6 cm} +o_K(1)n^{-d/\alpha}+o_{n,K}(1)n^{-d/\alpha}.\label{rustine6}
\end{align}
Thanks to (\ref{rustine6}), the end of the proof goes along the same lines as in the proof of Lemma \ref{funlemma}.
\end{proof}

Thanks to this Lemma \ref{modififunlemma}, we can derive a generalized version of Theorem \ref{cv}  in the same way as before. All the proof is the same. We only have to replace Lemma \ref{funlemma} by Lemma \ref{modififunlemma}.
Moreover, proving Theorem \ref{charac} under the hypothesis $\mathcal{H}_3$ is very similar with proving the same theorem under hypotheses $\mathcal{H}_1$ or $\mathcal{H}_2$.
We only need slight modifications for the generalization of Lemma \ref{majoc1}. Let us adapt this lemma under the hypothesis $\mathcal{H}_3$.
\begin{lem} \label{modifimajoc1}
Let $A$ be a closed ball. Let us assume hypothesis $\mathcal{H}_3$. Let $\Theta$ be a cluster-invariant point process. Then,
$$
\displaystyle\underset{M\rightarrow+\infty}\lim \hspace{0.2 cm} \underset{n\rightarrow+\infty}\limsup \hspace{0.2 cm}\P\left(\int_{B(0,Mn^{1/\alpha})^c}\mathbf{P}(Z_n(A-x)\geq 1)\Theta(dx)\geq M^{-\alpha/4} \right)=0.$$
\end{lem}
By a careful reading, one can convince oneself that one can prove a generalized version of Theorem \ref{charac} thanks to this generalized Lemma \ref{modifimajoc1}.
This Lemma \ref{modifimajoc1} is a direct consequence of Lemma \ref{tightness} and the following Lemma \ref{modifirecu} which is a generalized version of Lemma \ref{recu}.
\begin{lem}\label{modifirecu}
We assume hypothesis $\mathcal{H}_3$. Let $A$ be a closed ball.
There exists $C>0$ such that for every $M\geq1$, for every $n$ large enough and for every $x\in B(0,M n^{1/\alpha})^c$,\\
$$\mathbf{P}(Z_n(A-x)\geq 1)\leq C\frac{\mathbf{P}(Z_{M^{\alpha/2}n}(A-x)\geq 1)}{M^{\alpha/2}}.$$
\end{lem} 
\begin{proof}
As $d>\alpha/\beta$, following the same lines as in the proof of Lemma \ref{recu}, there exists a positive constant $C_{A,d,\alpha,\beta}$ such that for every $n\in\N^*$ and for every $x\in B(0,M n^{1/\alpha})^c$,
$$\begin{array}{ll}
\displaystyle\frac{\mathbf{P}\left( Z_n(A-x)\geq 1\right)}{\mathbf{P}\left( Z_{M^{\alpha/2}n}(A-x)\geq 1\right)}&\leq \frac{\displaystyle C_{A,d,\alpha,\beta}\P(n^{1/\alpha}\chi\in A-x)}{\displaystyle \P((M^{\alpha/2}n)^{1/\alpha}\chi\in A-x)}\\
&= C_{A,d,\alpha,\beta}\frac{\displaystyle\int_{(A-x)/n^{1/\alpha}}L_{\alpha}(y)dy}{\displaystyle\int_{(A-x)/(M^{\alpha/2}n)^{1/\alpha}}L_{\alpha}(y)dy}
\end{array}$$
where we recall that $\chi$ is a standard $d$-dimensional $\alpha$-stable law and $L_{\alpha}$ is its density.
According to \cite{BluGet}, there exists a constant $c_{\alpha,d}$ such that\\
$$L_{\alpha}(x)\underset{||x||\rightarrow +\infty}\sim c_{\alpha,d} ||x||^{-d-\alpha}.$$
However for $x\in B(0,Mn^{1/\alpha})^c$, $||n^{-1/\alpha}M^{-1/2}x||\geq \sqrt{M}$. Consequently, if $M\geq1$, we are allowed to use the equivalent just above.
Therefore, there exists a constant $C>0$ such that for every $n\in\N^*$ and for every $x\in B(0,M n^{1/\alpha})^c$,\\
$$\begin{array}{ll}
\displaystyle\frac{\mathbf{P}\left( Z_n(A-x)\geq 1\right)}{\mathbf{P}\left( Z_{M^{\alpha/2}n}(A-x)\geq 1\right)}
&\leq\displaystyle C\frac{\displaystyle\int_{(A-x)/n^{1/\alpha}}||y||^{-d-\alpha}dy}{\displaystyle\int_{(A-x)/(M^{\alpha/2}n)^{1/\alpha}}||y||^{-d-\alpha}dy}\vspace{0.2 cm}\\
&= C\times M^{-\alpha/2}.
\end{array}$$

\end{proof}
\section{Further discussion} \label{sec7}

Thanks to Theorem \ref{cv}, we know a local description of $\Lambda^{d,X}_{\infty}$ on every closed ball $A$. Indeed, if $P_A$ is a Poisson random variable of parameter $XI_A$, then $\Lambda^{d,X}_{\infty}(\cdot\cap A)$ can be constructed as a layering of $P_A$ independent copies of $\mathcal{N}_A$. A natural question is to know whether this is possible to describe $\mathcal{N_A}$ and $I_A$ in terms of $\Lambda_{\infty}^{d,X}$.
We prove that there exists another interesting description of $\Lambda_{\infty}^{d,X}$ which explains the meaning of $\mathcal{N}_A$ and $I_A$. Let us define $\mathbf{W}$ the set of measures $W$ on the set $\mathbf{N}$ of point measures on $\R^d$ which satisfy:
\begin{enumerate}[a)]
\item For any closed ball $A$ of $\R^d$, $W\left(\{\mu\in\mathbf{N}: \mu(A)>0 \right\})<+\infty$.
\item $W(\{\emptyset\})=0$ where $\emptyset$ is the null point measure.
\end{enumerate}
If $W\in \mathbf{W}$, we can define a Poisson point process $\Xi$ on $\mathbf{N}$ with intensity measure $W$. $\Xi$ can be written as
$$\Xi=\sum\limits_{i\in I}\delta_{N_i}$$
where $\{N_i,i\in I\}$ is a countable collection of point measures. Then we are able to define a new point process $\mathcal{E}_W$ on $\R^d$ as
$$\mathcal{E}_W=\sum\limits_{i\in I} N_i$$
where the sum of point measures means that we overlay them. $\mathcal{E}_W$ is still a locally finite point process thanks to assumption $a)$. We mention that the point process $\mathcal{E}_W$ is introduced in subsection 2.1 of \cite{MKM} in a slightly different (but equivalent) way.
\begin{prop}\label{global}
We assume hypothesis $\mathcal{H}_1$, $\mathcal{H}_2$ or $\mathcal{H}_3$. Then for every $\theta\in\R_+^*$, there exists a unique measure $W^{d,\theta}\in\mathbf{W}$ such that:
\begin{enumerate}[(i)]
\item $$\mathcal{E}_{W^{d,\theta}}\overset{law}=\Lambda_{\infty}^{d,\theta}. $$
\item For any closed ball $A$, $W^{d,\theta}(\{\mu\in\mathbf{N}:\mu(A)\geq 1\})=\theta I_A$.
\item For any closed ball $A$, for any borel set $U$ of the set of point measures on $A$, 
$$W^{d,\theta}\Bigg(\{\mu:\mu(\cdot\cap A)\in U\}\Bigg|\{\mu\in\mathbf{N}:\mu(A)\geq 1\}\Bigg)=\mathbb{P}\left(\mathcal{N}_A\in U\right).$$
\item Let $f\in\mathcal{F}_c^+(\R^d)$ whose support is a closed ball $A$. Then for every interval $I\subset \R_+^*$, $$W^{d,\theta}\left(\left\{\mu\in\mathbf{N}:\int f(x) d\mu(x)\in I \right\}\right)=\theta I_A\P\left(\int f(x)d\mathcal{N}_A(x)\in I \right).$$
\end{enumerate}
\end{prop}
\begin{proof}
$\Lambda_{\infty}^{d,\theta}$ is infinitely divisible. Indeed, it is obtained through a branching process starting from a Poisson point process which is infinitely divisible. Then, Theorem 2.1.10 in \cite{MKM} gives (i). By the local description of $\Lambda_{\infty}^{d,\theta}$ given in Theorem \ref{cv}, for any closed ball $A$, it holds that
\begin{align}
e^{-\theta I_A}&=\P\left(\Lambda_{\infty}^{d,\theta}(A)=0 \right)\nonumber\\
&=\P\left(\mathcal{E}_{W^{d,\theta}}(A)=0\right).\label{corfin}
\end{align}
Now, let us recall that $\mathcal{E}_{W^{d,\theta}}$ is defined through a Poisson point process $\Xi=\sum_{i\in I}\delta_{N_i}$ on $\mathbf{N}$ with intensity $W^{d,\theta}$. Then, from (\ref{corfin}), we deduce
\begin{align*}
e^{-\theta I_A}&=\P\left(N_i(A)=0,\forall i\in I \right)\\
&=\P\left(\Xi(\{\mu:\mu(A)\geq 1\})=0\right).
\end{align*}
Therefore, as $\Xi$ is a Poisson point process on $\mathbf{N}$ with intensity $W^{d,\theta}$,
\begin{align*}
e^{-\theta I_A}=\P\left(\Xi(\{\mu:\mu(A)\geq 1\})=0\right)=e^{-W^{d,\theta}(\{\mu:\mu(A)\geq 1\})}.
\end{align*}
This proves (ii). Now, let us consider a closed ball $A$ and $f\in\mathcal{F}_c^+(\R^d)$ such that $supp(f)\subset A$. For every point measure $\mu$, let us define $H_f(\mu)=\int f(x) d\mu(x)$. Then, by definition of $\mathcal{E}_{W^{d,\theta}}$.
\begin{align}
\E\left(\exp\left(-\int f(x) d\mathcal{E}_{W^{d,\theta}}(x) \right) \right)&=\E\left[\exp\left(-\int H_f(\mu)d\Xi(\mu) \right)\right].\label{identitefinale1}
\end{align}
However, $\Xi$ is a Poisson point process with intensity $W^{d,\theta}$. Therefore, by Lemma 3.1 in \cite{Kallenberg_modern}, (\ref{identitefinale1}) yields
\begin{align}
\E\left(\exp\left(-\int f(x) d\mathcal{E}_{W^{d,\theta}}(x) \right) \right)&=\exp\left(-\int (1-\exp(-H_f(\mu))dW^{d,\theta}(\mu) \right)\nonumber\\
&=\exp\left(\int \exp\left(-\int f(x) d\mu(x)\right) -1 dW^{d,\theta}(\mu) \right)\nonumber\\
&=\exp\left(\int \left(\exp\left(-\int f(x) d\mu(x)\right) -1\right)\textbf{1}\{\mu(A)\geq 1\} dW^{d,\theta}(\mu) \right)\nonumber
\end{align}
where the last equality stems from the fact that $supp(f)\subset A$. By (ii), this implies that
\begin{align}
\E\left(\exp\left(-\int f(x) d\mathcal{E}_{W^{d,\theta}}(x) \right) \right)&\nonumber\\
&\hspace{-4.5 cm}=\exp\Bigg(\theta I_A\Bigg(\int \exp\left(-\int f(x)d\mu(x)\right)dW^{d,\theta}(\mu|\mu(A)\geq 1) -1\Bigg) \Bigg).\label{identitefinale2}
\end{align}
However, by (i), we know that
\begin{align}
\E\left(\exp\left(-\int f(x) d\mathcal{E}_{W^{d,\theta}}(x) \right) \right)=\E\left(\exp\left(-\int f(x)d\Lambda_{\infty}^{d,\theta}(x)\right)\right). \nonumber
\end{align}
Consequently, by Proposition \ref{laplace}, we get
\begin{align}
\E\left(\exp\left(-\int f(x) d\mathcal{E}_{W^{d,\theta}}(x) \right) \right)=\exp\Bigg(\theta I_A\Bigg( \E\bigg[ \exp\left(-\int f(x)d\mathcal{N}_A(x) \right)\bigg]-1\Bigg) \Bigg).\label{identitefinale3}
\end{align}
Equalizing (\ref{identitefinale2}) and (\ref{identitefinale3}), we get
\begin{align}
\int\exp\left(-\int f(x)d\mu(x)\right)dW^{d,\theta}(\mu|\mu(A)\geq 1) =\E\bigg[ \exp\left(-\int f(x)d\mathcal{N}_A(x) \right)\bigg].
\end{align}
This proves (iii). Moreover, (iv) is a direct consequence of (ii) and (iii).

\end{proof}
An interesting property of the measure $W^{d,\theta}$ is that it is $\sigma$-finite but not finite. This is proved by the following proposition. We assume $\mathcal{H}_1$ for clarity of the discussion.
\begin{prop} \label{I}
Assuming hypothesis $\mathcal{H}_1$, $I_{B(0,r)}$ converges toward infinity when $r$ is going to infinity.
\end{prop}
\begin{proof}
Let $A$ be a closed ball.
Using the notation introduced in definition \ref{defiprat}, we recall that\\

$$\displaystyle I_{A}=\int_{A}\E_{\mathbf{Q}^*}\left[ \frac{1}{1+\sum\limits_{k=1}^{+\infty}\tilde{Y}_k(y-S_{w_k})}\right] dy.$$
Therefore, by Jensen's inequality,
\begin{align}
I_A&\displaystyle \geq \int_A \frac{1}{\E_{\mathbf{Q}^*}\left[1+ \sum\limits_{k=1}^{+\infty}\tilde{Y}_k(y-S_{w_k})\right]}dy.\label{inegaf1}
\end{align}
However, with the notation of the proof of Lemma \ref{funlemma}, for every $k\in\N^*$,
$$\E_{\mathbf{Q}^*}\left[\tilde{Y}_k(y-S_{w_k}) \right]=\E_{\mathbf{Q}^*}\left[\sum\limits_{u\in B(w_{k+1})} Z^u_{k-1}(A-y+S_{w_k}-\rho_u) \right].$$
Then, by construction of $\mathbf{Q}^*$, it holds that
\begin{align}
\E_{\mathbf{Q}^*}\left[\tilde{Y}_k(y-S_{w_k}) \right]=\E_{\mathbf{Q}^*}\left[|B(w_{k+1})| \right]\E_{\mathbf{Q}^*}\left[\E_{\mathbf{P}}\left[Z_{k-1}(A-y+z) \right]|_{z=S_{w_k}-\rho}\right] \label{inegaf2}
\end{align}
where $\rho$ is independent of $S_{w_k}$. Moreover $\rho$ is a standard Gaussian random variable because we assume hypothesis 
$\mathcal{H}_1$. We know that, under $\mathbf{Q}^*$, $1+|B(w_{k+1})|$ has law $\nu$ which has expectation $1+\sigma^2$. Moreover, the branching process is critical under $\mathbf{P}$. Thus, $\E_{\mathbf{P}}\left[Z_{k-1}(A-y+z) \right]=\P(\hat{S}_{k-1}\in A-y+z)$ where $\hat{S}$ is a brownian motion. Therefore, together with (\ref{inegaf2}), this yields
\begin{align}
\E_{\mathbf{Q}^*}\left[\tilde{Y}_k(y-S_{w_k}) \right]&=\sigma^2\P(\hat{S}_{2k}\in A-y)\nonumber\\
&=\sigma^2(4\pi k)^{-d/2}\int_A\exp\left(-||x-y||^2/4k \right)dx. \label{inegaf3}
\end{align}
Consequently, combining (\ref{inegaf1}) and (\ref{inegaf3}), we get
\begin{align}
I_A&\geq \displaystyle\int_A\left(\displaystyle 1+\sigma^2\sum\limits_{k=1}^{+\infty}(4\pi k)^{-d/2}\int_A \exp\left(-||x-y||^2/(4k)\right)dx\right)^{-1} dy. \label{inegaf4}
\end{align}
Now, if $A=B(0,r)$, a change of variable in (\ref{inegaf4}) yields
\begin{align}
I_{B(0,r)}&\geq r^d\int_{B(0,1)}\left(\displaystyle 1+\sigma^2\sum\limits_{k=1}^{+\infty}(4\pi k)^{-d/2}r^d\int_{B(0,1)} \exp\left(-r^{2}||x-y||^2/(4k)\right)dx\right)^{-1} dy\nonumber\\
&=\int_{B(0,1)}\left(\displaystyle r^{-d}+\sigma^2\sum\limits_{k=1}^{+\infty}(4\pi k)^{-d/2}\int_{B(0,1)} \exp\left(-r^{2}||x-y||^2/(4k)\right)dx\right)^{-1} dy. \label{inegaf5}
\end{align}

However, by the dominated convergence theorem, as $d\geq 3$, for every $y\in B(0,1)$,\\
$$\underset{r\rightarrow +\infty}\lim\hspace{0.2 cm}\displaystyle r^{-d}+\sigma^2\sum\limits_{k=1}^{+\infty}(4\pi k)^{-d/2}\int_{B(0,1)} \exp\left(-r^{2}||x-y||^2/(4k)\right)dx=0.$$

Consequently, using Fatou's Lemma in (\ref{inegaf5}),

$$\underset{r\rightarrow +\infty}\lim I_{B(0,r)}=+\infty.$$
\end{proof}
\section{Acknowledgments}
I would like to thank my Ph.D supervisor, Xinxin Chen, for her very useful pieces of advice without which this work could not have been carried out. I also want to thank Christophe Garban and Hui He whose fruitful discussions with Xinxin Chen have enlightened us a lot.
\bibliographystyle{alpha}
\bibliography{Bibli}
\end{document}